\documentclass[12pt]{article}
\usepackage{amsmath}
\usepackage{amsfonts}
\usepackage{amssymb}
\usepackage{amsthm}
\usepackage{fullpage}
\usepackage{url}

\newtheorem{theorem}{Theorem}
\newtheorem{lemma}{Lemma}

\newcommand{\be}{\begin{equation}}
\newcommand{\ee}{\end{equation}}

\renewcommand{\mid}{\,:\,}

\begin{document}
\title{Convergent relaxations of \\polynomial optimization problems\\ with
non-commuting variables}
\author{S. Pironio$^1$
\and M. Navascu\'es$^2$
\and A. Ac\'{\i}n$^3$
\and
\parbox{\linewidth}{\center\normalsize
$^1$Group of Applied Physics, University of Geneva\\
$^2$Institute for Mathematical Sciences, Imperial College London\\
$^3$ICFO-Institut de Ci\`{e}ncies Fot\`{o}niques}}

\date{}
\maketitle
\begin{abstract}
We consider optimization problems with polynomial inequality
constraints in non-commuting variables. These non-commuting
variables are viewed as bounded operators on a Hilbert space whose
dimension is not fixed and the associated polynomial inequalities as semidefinite positivity constraints. Such problems arise naturally in quantum theory and quantum information science. To solve them, we introduce a hierarchy of semidefinite programming relaxations which generates a monotone sequence of
lower bounds that converges to the optimal solution. We also introduce a criterion to
detect whether the global optimum is reached at a given relaxation step and show how to extract a global optimizer from the solution of the corresponding semidefinite programming problem.
\end{abstract}

\vfill
\pagebreak

\section{Introduction}
A standard problem in optimization theory is to find the global
minimum of a polynomial on a set constrained by polynomial
inequalities, that is, to solve the program
\begin{equation}\label{polyprog}
\begin{array}{llclll}
p^\star &=&\displaystyle\min_{x\in\mathbb{R}^n}&p(x)\\
   & &\text{s.t.}& q_i(x)\geq 0&\qquad i=1,\ldots,m\,,
\end{array}
\end{equation}
where $p(x)$ and $q_i(x)$ are real-valued polynomials in the
variable $x\in\mathbb{R}^n$. To deal with such non-convex problems,
Lasserre \cite{lasserre} introduced a sequence of semidefinite programming (SDP)\footnote{See Appendix~A for a brief introduction to semidefinite programming.} relaxations of increasing size, whose optima
converge monotonically to the global optimum $p^\star$; a similar
approach has been proposed by Parrilo \cite{Parrilo}.
This paper presents a generalization of Lasserre's method for a
non-commutative version of the optimization problem
(\ref{polyprog}). That is, we consider a polynomial optimization
problem where the variables $x=(x_1,\ldots,x_n)$ are not simply
real numbers, but non-commuting (NC) variables for which, in
general, $x_ix_j\neq x_j x_i$. Our motivation comes from quantum
theory, whose basic objects are matrices and operators that do not
commute. But our approach might also find applications in other
fields that involve optimization over matrices or operators, such as
in systems engineering \cite{helton-putinar}.

To write down the non-commutative version of $(\ref{polyprog})$, let $p(x)$ and $q_i(x)$ be polynomial expressions in the non-commuting variables $x=(x_1,\ldots,x_n)$. Given an Hilbert space $H$ and a set $X=(X_1,\ldots,X_n)$ of bounded operators on
$H$, we define operators $p(X)$ and $q_i(X)$ by substituting the variables $x$ by the operators $X$ in the expressions $p(x)$ and $q_i(x)$. Given in addition a normalized vector $\phi$ in $H$, we evaluate the polynomial $p(X)$  as $\langle \phi, p(X) \phi\rangle$.
The non-commutative version of the optimization problem
(\ref{polyprog}) considered here is then
\begin{equation}\label{polyncprog}
\begin{array}{llcll}
p^\star&=&\displaystyle\min_{(H,X,\phi)}& \multicolumn{1}{l}{\langle \phi, p(X) \phi\rangle}\\
&&\text{s.t.}&q_i(X)\succeq 0&\quad i=1,\ldots,m\,, \\
\end{array}
\end{equation}
where $q_i(X)\succeq 0$ means that the operator $q_i(X)$ should be positive semidefinite. In other words, given the input data $p(x)$ and $q_i(x)$,
 we look for the combination $(H,X,\phi)$ of Hilbert space $H$, operators $X$, and normalized state $\phi$ (both defined on $H$) that minimizes $\langle \phi, p(X) \phi\rangle$ subject to the constraints $q_i(X)\succeq 0$. It is important to note that the dimension of the Hilbert space $H$ is not fixed, but subject to optimization as well.

Taking inspiration from Lasserre's method \cite{lasserre} and from
the papers \cite{mpa,mpa2}, we introduce a hierarchy of SDP
relaxations for the above optimization problem. The optimal
solutions of these relaxations form a monotonically increasing
sequence of lower bounds on the global minimum $p^\star$. We prove
that this sequence converges to the optimum $p^\star$ when the set
of constraints $q_i(X)\succeq 0$ is such that every tuple of
operators $X=(X_1,\ldots,X_n)$ satisfying them are bounded, i.e.,
such that they satisfy $C^2-(X_1+\cdots+X_n)\succeq 0$ for some real constant $C>0$. Our proof is
constructive: from the sequence of optimal solutions of the SDP
relaxations, we build an explicit global minimizer
$(H^\star,X^\star,\phi^\star)$ for~(\ref{polyncprog}), where
$H^\star$ is, in general, infinite-dimensional. In some cases, the
SDP relaxation at a given finite step in the hierarchy may already
yield the global minimum $p^\star$. We introduce a criterion to
detect such events, and show in this case how to extract the
global minimizer $(H^\star,X^\star,\phi^\star)$ from the solution of
this particular SDP relaxation. The resulting Hilbert space
$H^\star$ is then finite-dimensional, with its dimension determined by
the rank of the matrices involved in the solution of the SDP
relaxation.

Our method can find direct applications in quantum information science, e.g. to compute upper-bounds on the maximal violation of Bell inequalities, and in quantum chemistry to compute atomic and molecular ground state energies. Practice reveals that convergence is usually fast and finite (up-to
machine precision).

\subsection{Relation to other works}
Unconstrained NC polynomial optimization problems (i.e. the minimization of a single polynomial $p(X)$ with no constraints of the form $q_i(X)\succeq 0$) were considered in~\cite{klep}. Such problems can also be solved using SDP, as implemented in the MATLAB toolbox \verb"NCSOStools"~\cite{cafuta}. Unlike constrained NC optimization~(\ref{polyncprog}), which requires a sequence of SDPs to compute the minimum, for unconstrained NC optimization a \emph{single} SDP is sufficient by a theorem of Helton that a symmetric NC polynomial is positive if and only if it admits a sum of square decomposition~\cite{helton}. This single SDP corresponds actually to the first step of our hierarchy when neglecting constraints coming from the conditions $q_i(X)\succeq 0$.

In \cite{mpa}, a subclass of the general constrained NC problem~(\ref{polyncprog}) which is of interest in quantum information (see (\ref{bell}) later in the text) was considered and a sequence of SDP relaxations introduced for it. The convergence of this SDP sequence was established in \cite{mpa2,doherty,ito}. Our work can be seen as a generalization of these results to arbitrary NC polynomial optimization.

In the commutative case, the convergence of the relaxations introduced by Lasserre is based on a sum of squares representation theorem of Putinar \cite{putinar} for positive polynomials. The
connection to Putinar's representation arises when considering the
dual problems of the SDP relaxations. A non-commutative analogue
of Putinar's result, the Positivstellensatz for non-commutative
positive polynomials, has been introduced by Helton and McCullough
\cite{heltonmc}. Although we first prove the convergence of the
hierarchy introduced here through the primal version of our SDP
relaxations (in the spirit of \cite{mpa2}) we also provide an alternative proof through the
duals, which exploits Helton and McCullough's result (as used in \cite{doherty} and \cite{ito}).

Note that the problem (\ref{polyncprog}) can also account for equality constraints $q_i(X)=0$, which can be enforced through the inequalities $q_i(X)\succeq 0$ and $-q_i(X)\succeq 0$. When constraints of the form $x_ix_j-x_jx_i$ are explicitly added to (\ref{polyncprog}), that is, when we require that the variables $x$ commute, our method reduces to the one introduced by Lasserre. It is interesting to note that the results presented here, such as the convergence of the hierarchy or the criterion to detect optimality, are easier to establish in the general non-commutative framework than they are in the specialized commutative case. This commutative setting has generated quite a large literature and the properties of the corresponding SDP relaxations have been thoroughly investigated. We refer to \cite{laurent} for a review. Our work provides a NC analogue only of the most basic results in the commutative case. It would be interesting to reexamine from a NC perspective other topics in this subject.

\subsection{Organization of the paper}
In Section~2, we define some notation and introduce in
more detail the class of problems that we consider here. Section~3 contains our main results: we introduce our hierarchy of SDP relaxations,
prove its convergence, show how to detect optimality at a
finite step in the hierarchy and how to extract a global optimizer. We then explain the relation between our approach and the
works of Helton and McCullough. We proceed by mentioning briefly how to modify our method to deal
efficiently with equality constraints. In particular, we discuss how it can be simplified when dealing with hermitian variables and how it reduces to Lasserre's method in the case of commuting variables. We end Section~3 by showing how our method can be extended to solve a slightly more general class of optimization problems. In Section~4, we illustrate our method on concrete examples. Finally, we briefly discuss practical applications of our method in the quantum setting in Section~5.

\section{Notation and definitions}
Let $\mathbb{K}\in\{\mathbb{R},\mathbb{C}\}$ denote the field of real or complex
numbers. We consider the algebra $\mathbb{K}[x,x^*]$ of polynomials in the $2n$ noncommuting variables $x=(x_1,\ldots,x_n)$ and $x^*=(x^*_1,\ldots,x^*_n)$ with coefficients from $\mathbb{K}$. That is, an element $p\in \mathbb{K}[x,x^*]$ is a linear combination
\begin{equation}
p=\sum_w p_w\, w
\end{equation}
of words $w$ in the $2n$ letters $x$ and $x^*$, where the sum is finite and $p_w\in \mathbb{K}$. We interpret $\ast$ as an involution (that
is, loosely speaking, a conjugate transpose) defined as follows: on letters, $(x_i)^*=x_i^*$ and $(x_i^*)^*=x_i$; on a word $ w = w _1\ldots  w _n$, $ w ^*= w _n^*\ldots  w _1^*$; and on a polynomial, $p^*=\sum_ w  p^*_ w   w ^*$, where $p^*_ w $ is the complex conjugate of $p_ w $. Thus $\mathbb{K}[x,x^*]$ is the free $\ast$-algebra generated by the $2n$ variables $x_1,\ldots,x_n,x_1^*,\ldots,x_n^*$.
In the following, we will often view these $2n$ variables as $x_1,\ldots,x_n,\linebreak[1]x_{n+1},\linebreak[1]\ldots,x_{2n}$ by identifying $x_{n+i}$ with $x_i^*$.

Throughout this paper, the symbols $u,v,w$ always denote words and we denote the empty word by $1$. We use the notation $\mathcal{W}_d$ for the set of all words of length $|w|$ at most $d$, and $\mathcal{W}_\infty$ for the set of all words (of unrestricted length). The number of words in  $\mathcal{W}_d$ is $|\mathcal{W}_d|=\left((2n)^{d+1}-1\right)/(2n-1)$.
The degree of a polynomial $p$ is the length of the longest word in $p$ and is denoted $\text{deg}(p)$. We let $\mathbb{K}[x,x^*]_d$ denote the set of polynomials $p=\sum_{| w |\leq d}  p_ w \,  w $ of degree $\leq d$. If necessary, a polynomial of degree
$d$ can be viewed as a polynomial of higher degree $d'$ by setting to zero the coefficients associated with words of length larger than $d$.
A polynomial $p$ is said to be hermitian if $p^*=p$, or in term of its coefficients, if $p^*_ w =p_{ w ^*}$.
Note that words can be interpreted as monomials and we will sometimes use the two terms interchangeably. We will then also refer to the length $|w|$ of a word as the degree of the monomial $w$ and to $\mathcal{W}_d$ as a monomial basis for $\mathbb{K}[x,x^*]_d$.

Let $\mathcal{B}(H)$ denote the set of bounded operators on a Hilbert space $H$ defined on the field $\mathbb{K}$.
Consider a set of operators $X=(X_1,\ldots,X_n)$ from $\mathcal{B}(H)$. Given the polynomial $p\in \mathbb{K}[x,x^*]$, we define the operator
$p(X)\in \mathcal{B}(H)$ by substituting every variable $x_i$ by the operator $X_i$ and every variable $x_i^*$ by $X_i^*$, where $^*$ denotes the adjoint operation on $H$.  If $p^*=p$ is a
hermitian polynomial, then $p(X)=p^*(X)$ is a hermitian
operator and the quantity $\langle \phi, p(X)\phi \rangle$ is
real for every vector $\phi$ in $H$. A hermitian operator $O$
is said to be positive semidefinite, a fact that we denote by
$O\succeq 0$, if $\langle \phi, O \phi \rangle \geq 0$
for all $\phi\in H$.

\subsection{Formulation of the optimization problem}
Let $p$ and $q_i$ $(i=1,\ldots,m)$ be hermitian polynomials in
$\mathbb{K}[x,x^*]$. We are interested in the following
optimization problem:
\begin{equation}\label{polyncprog2}
\mathbf{P}:\qquad
\begin{array}{llcll}
\quad p^\star&=&\displaystyle\min_{(H,X,\phi)}& \multicolumn{1}{l}{\langle \phi, p(X) \phi\rangle}\\
&&\text{s.t.}& q_i(X)\succeq
0&\quad i=1,\ldots,m\,,\\
\end{array}
\end{equation}
where the optimization should be understood over all Hilbert spaces
$H$ (of arbitrary dimension), all sets of operators
$X=(X_1,\ldots,X_n)$ in $\mathcal{B}(H)$, and all normalized vectors
$\phi$ in $H$.  We assume throughout the remaining of the paper that
this problem admits a feasible solution, that is, that there
exists a triple $(H,X,\phi)$ such that $\langle\phi,\phi\rangle=1$
and $q_i(X)\succeq 0$ for $i=1,\ldots,m$.

Let $Q=\{q_i\mid i=1,\ldots,m\}$ be the set of polynomials determining the positivity constraints in (\ref{polyncprog2}). The following definitions follow those used in~\cite{klep2}. The \emph{positivity domain} $\mathbf{S}_Q$ associated to $Q$ is the class of tuples $X=(X_1,\ldots,X_n)$ of bounded operators on a Hilbert space making each $q_i(X)$ a positive semidefinite operator.
The \emph{quadratic module} $\mathbf{M}_Q$ is the set of all elements of the form $\sum_if_i^*f_i+\sum_i\sum_j g_{ij}^*q_ig_{ij}$ where $f_i$ and $g_{ij}$ are polynomials in $\mathbb{K}[x,x^*]$.
We say that $\mathbf{M}_Q$ is Archimedean if there exists a real constant $C$ such that $C^2-(x_1^*x_1+\cdots+x_{2n}^*x_{2n})\in \mathbf{M}_Q$. In this case, the positivity domain $\mathbf{S}_Q$ is bounded: for all $X \in \mathbf{S}_Q$, $C^2-(X_1^*X_1+\cdots+X^*_{2n}X_{2n})\succeq 0$. Note that if $\mathbf{S}_Q$ is bounded, we can always add $C^2-(x_1^*x_1+\cdots+x_{2n}^*x_{2n})$ to $Q$ for a sufficiently large $C$ to make $\mathbf{M}_Q$ Archimedean without changing $\mathbf{S}_Q$. In the following we will always assume that the constraints in $Q$ are such that $\mathbf{M}_Q$ is Archimedean.

\section{Main results}\label{relaxsec}
\subsection{Moment and localizing matrices}
Let $y=(y_w )_{|w|\leq d}\in \mathbb{K}^{|\mathcal{W}_d|}$ be a sequence of real or complex numbers indexed in $\mathcal{W}_d$, i.e., to each word $w\in \mathcal{W}_d$ corresponds a number $y_w\in \mathbb{K}$. We define the linear mapping $L_y:\mathbb{K}[x,x^*]_d\mapsto \mathbb{K}$ as
\begin{equation}
\label{map} p\mapsto L_y(p)=\sum_{|w|\leq d} p_w  y_w \,.
\end{equation}
By analogy with \cite{lasserre}, given a sequence $y=(y_ w )_{| w |\leq 2k}$ indexed in $W_{2k}$, we define the \emph{moment matrix}
$M_k(y)$ of order $k$ as a matrix with rows and columns indexed in $W_k$ and whose entry $(v,w)$ is given by
\begin{equation}\label{mom}
M_k(y)(v,w)=L_y( v ^* w )=y_{ v ^* w }.
\end{equation}
If $q=\sum_{|u|\leq d}q_u u$ is a polynomial of degree $d$ and
$y=(y_ w )_{| w |\leq 2k+d}$ a sequence indexed in $\mathcal{W}_{2k+d}$, we define the \emph{localizing matrix} $M_k(q y)$ as the matrix with
rows and columns indexed in $\mathcal{W}_k$, and whose entry $(v ,w )$ is
\begin{equation}\label{loc}
M_k(qy)( v , w )=L_y( v^*  q  w )=\sum_{| u |\leq d} q_ u
y_{ v ^* u  w }\,.
\end{equation}
We say that a sequence $y=(y_ w )_{| w |\leq 2k}$ admits a \emph{moment representation}, if there exists a triple $(H,X,\phi)$ with a normalized $\phi$ such that
\begin{equation}\label{momrepr}
y_ w =\langle \phi, w (X)\phi\rangle\,,
\end{equation}
for all $| w |\leq 2k$.
\begin{lemma}\label{posmom} Let $y=(y_ w )_{| w |\leq 2k}$ be a sequence admitting a moment representation. Then $y_1=1$ and $M_k(y)\succeq 0$. If the moment representation (\ref{momrepr}) is such that $q(X)\succeq 0$ for some $q\in\mathbb{K}[x,x^*]$, then  in addition $M_{k-d}(qy)\succeq 0$, where $d=\lceil \text{deg}(q)/2\rceil$.
\end{lemma}
\begin{proof}
Eq. (\ref{momrepr}) immediately implies $y_1=1$ since $\langle\phi,\phi\rangle=1$. The positivity of the moment matrix $M_k(y)$ follows from the fact that for any vector $z\in
\mathbb{K}^{|\mathcal{W}_k|}$
\begin{eqnarray}
z^* M_k(y)z&=&\sum_{ v,w} z^*_v
M_k(y)( v,w)z_ w =\sum_{v,w}  z^*_ v
y_{ v ^* w }
z_ w \nonumber\\
&=&\langle \phi,\sum_ v  z^*_ v   v ^*(X) \sum_ w
z_ w   w(X)\, \phi\rangle=\langle\phi, z^*(X)
z(X)\,\phi\rangle\geq 0\,,
\end{eqnarray}
where we have defined the operator $z(X)=\sum_ w  z_ w   w (X)$.

Suppose now that $y$ admits a moment representation (\ref{momrepr}) by a triple
$(H,X,\phi)$ such that $q(X)\succeq 0$. Then $M_{k-d}(qy)\succeq 0$ since for all vectors $z\in \mathbb{K}^{|\mathcal{W}_{k-d}|}$,
\begin{eqnarray}
z^* M_{k-d}(qy)z&=&\sum_{ v , w } z^*_ v
M_{k-d}(qy)( v , w ) z_ w =\sum_{ v , w , u }z^*_ v  q_ u  y_{ v ^* u w }
z_ w \nonumber\\
&=&\langle\phi,\sum_ v  z^*_ v   v ^*(X)
\sum_ u  q_ u   u (X) \sum_ w  z_ w
 w (X)\,\phi\rangle\nonumber\\
 &=&\langle\phi, z^*(X) q(X) z(X)\,\phi\rangle\geq
0\,,
\end{eqnarray}
where the last inequality follows from the fact that $q(X)\succeq
0$.
\end{proof}

\subsection{Convergent SDP relaxations}\label{csr}
For $2k\geq \max \left\{\mathrm{deg}(p), \max_i\mathrm{deg}(q_i)\right\}$,
consider the SDP problem
\begin{equation}\label{relax}
\mathbf{R_k}:\qquad
\begin{array}{ccllr}
\quad p^{k}&=&\displaystyle\min_{y} &\sum_ w  p_ w  y_ w \\
&&\mathrm{s.t.} & y_1=1\\
&& & M_k(y)\succeq 0\\
&& & M_{k-d_i}(q_iy)\succeq 0& \quad i=1,\ldots,m,\\
\end{array}
\end{equation}
where $d_i=\lceil\text{deg}(q_i)/2\rceil$ and the optimization is
over $y=(y_ w )_{| w |\leq 2k}\in \mathbb{K}^{|\mathcal{W}_{2k}|}$. The optimum $p^{k}$ provides a lower-bound on the global optimum $p^\star$ of the original problem $\mathbf{P}$, since any feasible solution $(H,X,\phi)$ of $\mathbf{P}$ yields a feasible solution $y$ of $\mathbf{R_k}$ through Eq.~(\ref{momrepr}) and Lemma~\ref{posmom}. We refer to $\mathbf{R_k}$ as the SDP relaxation of order $k$ of $\mathbf{P}$.
Since the positivity of the moment and localizing matrices of a given order $k'$ implies the positivity of the moment and localizing matrices of lower orders $k$, the sequences of SDP relaxations form a hierarchy in the sense that $p^k\leq p^{k'}$ when $k\leq k'$.

\begin{theorem}
If $\mathbf{M}_Q$ is Archimedean, $\lim_{k\rightarrow\infty} p^{k}=p^\star$.
\end{theorem}

Remember that if $\mathbf{M}_Q$ is Archimedean, there exists polynomials $f_{i}$ and $g_{ij}$ and a constant $C$ such that $C^2-(x_1^*x_1+\cdots+x_{2n}^*x_{2n})=\sum_if_i^*f_i+\sum_i\sum_j g_{ij}^*q_ig_{ij}$. In the following, we write $d_\mathbf{M}=\max_{ij}\{\text{deg}(f_i),\text{deg}(g_{ij})+d_i\}$. Note that $d_\mathbf{M}\geq 1$, with $d_\mathbf{M}=1$ when $C^2-(x_1^*x_1+\cdots+x_{2n}^*x_{2n})$ is contained in $Q$. Although the asymptotic behavior of the hierarchy of SDP relaxations only depends on the quadratic module being Archimedean, it may be advantageous in practice to add the constraint $C^2-(x_1^*x_1+\cdots+x_{2n}^*x_{2n})$ to $Q$. This will guarantee in particular that the first step of the hierarchy has a bounded solution (see Lemma~\ref{lempr2}).

The proof of Theorem~1 is based on the following four lemmas.

\begin{lemma}\label{lempr1}
Let $c=C^2-(x^*_1x_1+\cdots+x_{2n}^*x_{2n})$ and let $y$ be a sequence satisfying $y_1=1$, $M_k(y)\succeq 0$, and $M_{k-1}(cy)\succeq 0$. Then $|y_w|\leq C^{|w|}$ for all $|w|\leq 2k$.
\end{lemma}
\begin{proof}
The diagonal elements of $M_{k-1}(cy)$ are of the form
$C^2y_{ w ^* w }-\sum_{i=1}^{2n}y_{w^*x_i^*x_iw}$ with
$|w|\leq k-1$. Since the localizing matrix $M_{k-1}(cy)$ is positive
semidefinite, these diagonal entries must be positive, that is,
$\sum_{i=1}^{2n}y_{w^*x_i^*x_iw}\leq C^2y_{ w ^* w }$. In addition, it also holds that $y_{w^*x_i^*x_iw}\geq 0$ since these numbers are diagonal entries of the moment matrix $M_k(y)$. It thus follows that $y_{w^*x_i^*x_iw}\leq C^2y_{ w ^* w }$ for all $|w|\leq k-1$ and all $i=1,\dots,2n$. Given that $y_{1}=1$, we deduce by induction that
$y_{ w ^*w }\leq C^{2| w |}$ for all $| w |\leq k$.

The moment matrix $M_k(y)$ admits the following matrix
\be
\left(\begin{array}{cc}y_{ w ^*w }&y_{ w ^* v }\\
y_{ v ^* w }&y_{ v ^* v }\end{array}\right),
\ee
as a submatrix, where $| w |,| v |\leq k$. Since
$M_k(y)\succeq 0$, the above submatrix must also be positive
semidefinite, which is equivalent to the condition that
$y_{ w ^* v }y_{ v ^* w }\leq
y_{ w ^* w }y_{ v ^* v }$. Combining this relation
with the previous bound on $y_{ w ^* w }$ and the fact that
$y_{ v ^* w }=y^*_{w
^* v }$ which follows from
the hermicity of $M_k(y)$, we deduce that $|y_ w |\leq
C^{| w |}$ for all $| w |\leq 2k$.
\end{proof}
\begin{lemma}\label{lempr2}
Let $2k\geq \max \left\{\mathrm{deg}(p), \max_i\mathrm{deg}(q_i)\right\}$ and let $M_Q$ be Archimedean. Let $y$ be a feasible solution of the relaxation $\mathbf{R_{k-1+d_M}}$. Then $|y_w|\leq C^{|w|}$ for all $|w|\leq 2k$.
\end{lemma}
\begin{proof}
First note that if $f\in\mathbb{K}[x,x^*]_d$ is a polynomial of degree $d$ and $y$ a sequence such that $M_{k+d}(y)\succeq 0$, then $M_k(f^*fy)\succeq 0$. This follows from the fact that $M_{k+d}(y)\succeq 0$ implies $\sum_{|v|,|w|\leq k} \sum_{|\tilde v|,|\tilde w|\leq d} z^*_vf^*_{\tilde v} L_y(v^*\tilde v^* \tilde w w)z_wf_{\tilde w}\geq 0$ for all $z\in\mathbb{K}^{|\mathcal{W}_k|}$ and from the identity $\sum_{|v|,|w|\leq k} \sum_{|\tilde v|,|\tilde w|\leq d} z^*_vf^*_{\tilde v} L_y(v^*\tilde v^* \tilde w w)z_wf_{\tilde w}=\sum_{|v|,|w|\leq k}z^*_vL_y(v^*f^*fw)z_w=z^*M_{k}(f^*fy)z$. Similarly, if $g\in\mathbb{K}[x,x^*]_d$ is a polynomial of degree $d$ and $y$ a sequence such that $M_{k+d}(q_iy)\succeq 0$, then $M_k(g^*q_igy)\succeq 0$. Indeed, from $M_{k+d}(q_iy)\succeq 0$ we deduce that for all $z\in\mathbb{K}^{|\mathcal{W}_k|}$, $\sum_{|v|,|w|\leq k} \sum_{|\tilde v|,|\tilde w|\leq d} z^*_vg^*_{\tilde v} L_y(v^*\tilde v^* q_i\tilde w w)z_wg_{\tilde w}\geq 0$, and the left-hand side of this last inequality is equal to
$
\sum_{|v|,|w|\leq k}z^*_vL_y(v^*g^*q_igw)z_w=z^*M_{k}(g^*q_igy)z$.

Now, let $y$ be the optimal solution of the relaxation $\mathbf{R_{k-1+d_M}}$ as in the statement of the lemma and let $c=C^2-(x^*_1x_1+\cdots+x_{2n}^*x_{2n})$. As $\mathbf{M}_Q$ is Archimedean, we can write $c=\sum_i f_i^*f_i+\sum_{ij}g_{ij}^*q_ig_{ij}$, and thus $M_{k-1}(cy)=\sum_i M_{k-1}(f^*_if_iy)+\sum_{ij}M_{k-1}(g^*_{ij}q_ig_{ij}y)$. Since $M_{k-1+d_\mathbf{M}}(y)\succeq 0$ and $M_{k-1+d_\mathbf{M}-d_i}(q_iy)\succeq 0$, the argument outlined here above  implies that $M_{k-1}(f^*_if_iy)\succeq 0$ and $M_{k-1}(g^*_{ij}q_ig_{ij}y)\succeq 0$. This in turn implies $M_{k-1}(cy)\succeq 0$. From Lemma~\ref{lempr1}, we then deduce that $|y_w|\leq C^{|w|}$ for all $|w|\leq 2k$.
\end{proof}

\begin{lemma}\label{lempr3}
If $M_Q$ is Archimedean, the optima $p^k$ of the relaxations $\mathbf{R_k}$ form, for $k$ large enough, a monotically increasing bounded sequence. Therefore, the limit $\hat{p}=\lim_{k\rightarrow\infty} p^{k}$ exists.
\end{lemma}
\begin{proof}
Let $l=l'-1+d_\mathbf{M}$ with $2l'\geq \max \left\{\mathrm{deg}(p), \max_i\mathrm{deg}(q_i)\right\}$, and let $y$ be the solution of the relaxation $\mathbf{R_{l}}$ with objective value $p^{l}$. From Lemma~\ref{lempr2}, the entries $y_w$ with $|w|\leq 2l'$ are bounded, i.e., $|y_w|\leq C^{|w|}$. Thus the solution $p^l=\sum_{|w|\leq l'}p_wy_w$ is bounded as well. We also have that $p^\star$ is bounded since for $\mathbf{M_Q}$ Archimedean, the positivity domain $\mathbf{S_Q}$ is bounded. For all $k\geq l$, $p^{l}\leq p^k\leq p^{k+1}\leq p^\star$. Thus the $\left(p^{k}\right)_{k\geq l}$ form a monotonely increasing bounded sequence and the limit $\hat{p}=\lim_{k\rightarrow\infty}
p^{k}$ exists.
\end{proof}
\begin{lemma}\label{lempr4}
Let $\mathbf{M}_Q$ be Archimedean and let $\hat p=\lim_{k\rightarrow\infty}
p^{k}$ be the limit of the optimal solutions $p^{k}$ of the relaxations $\mathbf{R_k}$. Then there exists an infinite sequence $\hat y=(\hat y_ w )_{| w |=0,1,\ldots}$ indexed in $\mathcal{W}_\infty$ such that $|y_w|\leq C^{|w|}$,
\begin{eqnarray}\label{momconstr}
\sum_ w  p_ w  \hat y_ w &=&\hat{p}\,,\nonumber\\
\hat y_1 &=&1\,,
\end{eqnarray}
and
\begin{eqnarray}\label{momconstr2}
M_k(\hat y)&\succeq& 0\,,\nonumber\\
M_{k-d_i}(q_i\hat y)&\succeq& 0 \quad i=1,\ldots,m
\end{eqnarray}
for all $k$ large enough.
\end{lemma}
\begin{proof}
For any $k$ such that $2k\geq \max \left\{\mathrm{deg}(p),\mathrm{deg}(q_i)\right\}$, let $y^{k-1+d_\mathbf{M}}$ be a feasible solution of the relaxation $\mathbf{R_{k-1+d_\mathbf{M}}}$ with objective value $\hat p$. Such a
solution always exists because the problem $\mathbf{R_{k-1+d_\mathbf{M}}}$ is convex
and there exist feasible points of $\mathbf{R_{k-1+d_\mathbf{M}}}$ with optimal values
$p_1$ and $p_2$ satisfying $p_1\leq \hat{p}\leq p_2$ (take for
instance $p_1=p^{k-1+d_\mathbf{M}}$ and $p_2=p^\star$).
By Lemma~\ref{lempr2}, the entries $y^{k-1+d_\mathbf{M}}_ w $ with $| w |\leq 2k$ are bounded, i.e., $|y^{k-1+d_\mathbf{M}}_w|\leq C^{|w|}$. Let $\tilde y^k$ be the restriction of the solution $y^{k-1+d_\mathbf{M}}$ to the $|w|\leq 2k$. That is, $\tilde y^k=(y^{k-1+d_\mathbf{M}}_w)_{|w|\leq 2k}$ is the subsequence of $y^{k-1+d_\mathbf{M}}$ composed of the entries $y^{k-1+d_\mathbf{M}}_ w $ with $| w |\leq 2k$. Complete $\tilde y^k$ with zeros to make it an infinite vector ${y}^{k}$ in $l_\infty$ and perform the renormalization
$y^{k}_ w \to z^{k}_ w ={y^{k}_ w }/{C^{| w |}}$.
Each vector $z^{k}$ thus belongs to the unit ball of $l_\infty$,
and the sequence
$\left(z^{k}\right)_{k\geq l}$ admits by the
Banach-Alaoglu theorem a subsequence $(z^{k_i})_{i=1,2,\ldots}$ that converges
in the weak$\text{-}\nolinebreak[4]\ast$ topology to a limit $\lim_{i\to\infty}z^{k_i}=\hat z$ \cite{reedsimon}.
This implies in particular pointwise convergence, i.e.,
$\lim_{i\to\infty} z^{k_i}_ w =\hat z_ w $ for all $ w $.
Define the infinite vector $\hat y$ through $\hat y_ w = \hat
z_ w  C^{| w |}$. The pointwise convergence $z^{k_i}\to \hat
z$ implies the pointwise convergence of $y^{k_i}\to \hat y$, i.e.,
$\lim_{i\to\infty} y^{k_i}_ w =\hat y_ w $ for all $ w $.
Since $\sum_ w  p_ w  y^k_ w =\hat{p}$, $y^k_1=1$, $M_k(y^{k'})\succeq 0$,
and $M_{k-d_i}(q_iy^{k'})\succeq 0$ ($i=1,\ldots,m_q$) for all $k,k'$ with $k'\geq k$, we
deduce Eqs.~(\ref{momconstr}) and (\ref{momconstr2}) from the pointwise convergence of $y^{k_i}\to \hat y$.
\end{proof}
\begin{proof}[Proof of Theorem 1]
By Lemma~\ref{lempr3}, the limit $\hat{p}=\lim_{k\rightarrow\infty}
p^{k}$ exists. We obviously have that $\hat{p}\leq p^\star$. We
now show that there exist a set of operators $\hat{X}$ and a vector
$\hat\phi$ in a Hilbert space $\hat H$ (possibly of infinite
dimension) that yield a feasible solution of $\mathbf{P}$
with objective value $\hat{p}$. Thus, we also have that $\hat{p}\geq
p^\star$, and therefore $\hat{p}=p^\star$. Incidentally, this implies that the minimum appearing in equation (\ref{polyncprog2}) is well defined, i.e., it is not an infimum, as one would have expected in general.

To build $(\hat H,\hat X,\hat \phi)$, we perform a Gelfand-Naimark-Segal like construction. Let $\hat y$ be the infinite sequence defined in Lemma~\ref{lempr4}. Consider the linear functional $L_{\hat y}: \mathbb{K}[x,x^*]\mapsto \mathbb{K}$, $p\mapsto L_{\hat y}(p)=
\sum_ w  p_ w  \hat y_ w $. Since $M_k(\hat y)\succeq 0$ for all $k$, this linear functional is positive in the sense that $L_{\hat y}(p^*p)=\sum_{v,w} p^*_vL_{\hat y}(v^*w)p_w\geq 0$ for all $p$. It thus defines a semi-inner product on $\mathbb{K}[x,x^*]$ through
\begin{equation}\label{innerproduct}
\langle p,q\rangle=L_{\hat y}(p^*q)\,.
\end{equation}
Define the set
\begin{equation}
I=\{p\in \mathbb{K}[x,x^*]\mid \langle p,p\rangle=0\}\,.
\end{equation}
By the Cauchy-Schwarz inequality (which is valid for semi-inner products), the set $I$ is a linear subspace of $\mathbb{K}[x,x^*]$. Moreover, it is a left ideal of $\mathbb{K}[x,x^*]$. To show that $I$ is a left ideal of $\mathbb{K}[x,x^*]$, it is sufficient to show that $x_i I \subseteq I$ for all $i=1,\ldots,2n$. Since $\mathbf{M}_Q$ is Archimedean, there is some $C$ such that $c=C^2-\sum_i x_i^*x_i \in \mathbf{M}_Q$ and, as in the proof of Lemma~\ref{lempr2}, one can show that $M_k(c\hat y)\succeq 0$, from which it follows that
\begin{equation}\label{leftid}
0\leq L_{\hat y}(p^*cp)=C^2L_{\hat y}(p^*p)-\sum_iL_{\hat y}(p^*x_i^*x_ip).
\end{equation}
Since $L_{\hat y}(p^*x_i^*x_ip)\geq 0$ for all $i$, (\ref{leftid}) implies that
\begin{equation}\label{leftid2}
0\leq L_{\hat y}(p^*x_i^*x_ip)\leq C^2L_{\hat y}(p^*p)\,.
\end{equation}
For all $p\in I$, we thus have that $L_{\hat y}(p^*x_i^*x_ip)=0$, that is, $x_ip\in I$.

The definition (\ref{innerproduct}) of $\langle\cdot,\cdot\rangle$ induces a well defined inner product on the quotient $\mathbb{K}[x,x^*]/I$. Let $\hat H$ denote the Hilbert space corresponding to the completion of $\mathbb{K}[x,x^*]/I$ with respect to this scalar product.
We will now construct operators $\hat X$ on $\hat H$. For every $x_i$, let $\hat X_i$ be the operator of left multiplication by $x_i$ on $\mathbb{K}[x,x^*]/I$, i.e.,
\begin{equation}
\hat X_i  (p+I)=x_ip+I\,.
\end{equation}
Since $I$ is a left ideal, this map is well-defined for every $x_i$. It is linear, and by (\ref{leftid2}) it is bounded. Thus it extends uniquely to a bounded operator on $H$, which we denote by the same symbol $\hat X_i$.
Note that the map is also consistent with the involution on $\mathbb{K}[x,x^*]$, i.e., it satisfies $\hat X^*_i=\hat X_{i+n}$, since $\langle p,X^*_iq\rangle=\langle X_ip, q\rangle=\langle x_ip,q\rangle=L_{\hat y}(p^* x^*_iq)=L_{\hat y}(p^* x_{i+n}q)=\langle p,X_{i+n}q\rangle$.

Let $\hat\phi$ be the vector of $\mathbb{K}[x,x^*]/I$ corresponding to the identity polynomial $1$. The fact that $\hat y_1=1$ implies that the vector $\hat\phi$ is
normalized: $\langle \hat\phi,\hat\phi\rangle=1$. From (\ref{momconstr}), it follows that
\begin{eqnarray}
\langle \hat\phi,p(\hat{X})\hat\phi\rangle &=&\sum_ w  p_ w
\langle \hat\phi, w (\hat{X})\hat\phi\rangle =\sum_ w  p_ w
\langle 1,  w \rangle=\sum_ w  p_ w
\hat y_ w  =\hat{p}\,.
\end{eqnarray}
To show that $(\hat H, \hat X, \hat \phi)$ yields a feasible
solution to $\mathbf{P}$ with objective value $\hat{p}$, it
remains to show that the operators $\hat X$ satisfy $q_i(\hat
X)\succeq 0$ ($i=1,\ldots,m$), i.e., that
$\langle r,q_i(\hat X) r\rangle\geq 0$ for all $r\in \hat H$. But since any $r\in\hat H$ can be approximated to
arbitrary precision by elements of the pre-Hilbert space $\mathbb{K}[x,x^*]/I$, it is sufficient to show that
$\langle p,q_i(\hat X) p\rangle\geq 0$ for all $p\in \mathbb{K}[x,x^*]$. This follows from
\begin{equation}
\langle p,q_i(\hat X) p\rangle=\langle p,q_i p\rangle=L_{\hat y}(p^*q_i p)=\sum_{v,w}p^*_vL_{\hat y}(v^*q_iw)p_w\geq 0\,,
\end{equation}
since $M_{k-d_i}(q_i\hat y)\succeq0$ for all $k$.
\end{proof}

\subsection{Optimality detection and extraction of
optimizers}
\label{detectoptim}
In this subsection, we introduce a
criterion that allows to detect whether the relaxation of order $k$
already yields the optimal value $p^\star$. If so, it is possible to
extract a global optimizer $(H^\star,X^\star, \phi^\star)$ from the
optimal solution of this relaxation. The procedure to extract this
optimizer is described in the proof of the following theorem.
\begin{theorem}
Assume that the optimal solution $y^{k}$ of the relaxation of
order $k$ satisfies
\be
\mathrm{rank}\,M_k(y^{k})=\mathrm{rank}\,M_{k-d}(y^{k}),
\label{rankloop}\ee where
$d=\max_i d_i\geq 1$. Then
$p^{k}=p^\star$, i.e., the optimum of the relaxation of order $k$ is
the global optimum of the original problem \eqref{polyncprog2}.
Moreover, there exists a global optimizer
$(H^\star,X^\star,\phi^\star)$ of (\ref{polyncprog2}) with
$\dim{H^\star}=\mathrm{rank}\,M_{k-d}(y^{k})$.
\end{theorem}

\begin{proof}
We show that when (\ref{rankloop}) holds we can find a solution $(H,
X, \phi)$ to (\ref{polyncprog2}) with objective value $p^{k}$.
This implies that $p^{k}\geq p^\star$, and thus $p^{k}=p^\star$
since we also have $p^{k}\leq p^\star$.

Let $r=\mathrm{rank}\,M_k(y^{k})=\mathrm{rank}\,M_{k-d}(y^{k})$. Since the moment matrix $M_{k}({y}^{k})$ is positive semidefinite, it admits a Gram decomposition. That is, to each row (and column), indexed by a $w$ with $|w|\leq k$, can be  associated a vector $\overline{w}\in\mathbb{K}^{r}$ such that $M_k(y^{k})(w, v )=y^{k}_{w^* v }=\langle \overline{w},\overline{ v }\rangle$. We define the Hilbert space $H$ as
$H=\text{span}\{\overline{w}\mid| w |\leq k\}$, with dimension $\dim{H}
= r$. Note that (\ref{rankloop}) implies that
\be H=\mbox{span}\{\overline{w}\mid| w |\leq
k\}=\mbox{span}\{\overline{w}\mid| w |\leq k-d\}.
\label{vectorial}\ee
We now define $2n$ linear operators $X_i$ through their actions on the $\overline{w}$'s with $|w|\leq k-1$ in the following way
\begin{equation}\label{defxopt}
X_i\overline{w}=\overline{x_iw}\,.
\end{equation}
Note that when $d\geq 1$, the operators are well defined on the whole space $H$ since by (\ref{vectorial}) the set of vectors $\overline{w}$ with $|w|\leq k-d\leq k-1$ span $H$. This definition is also consistent in the sense that if $\overline{f}\in H$ admits two different decompositions $\overline{f}=\sum a_w\overline{w}=\sum b_w\overline{w}$ as a linear combination of the vectors $\{\overline{w}:|w|\leq
k-1\}$, then $\sum a_w\overline{x_iw}=\sum b_w\overline{x_iw}$.
Indeed the following equality
\begin{eqnarray}
\langle\overline{v},\sum_w(a_w-b_w)\overline{x_iw}\rangle&=&\sum_w (a_w-b_w)y_{v^*x_iw}=\sum_w (a_w-b_w)y_{(x_i^*v)^*w}\nonumber\\
&=&\langle\overline{x_i^*v},\sum_w (a_w-b_w)\overline{w}\rangle=\langle\overline{x_i^*v},0\rangle=0\,,
\end{eqnarray}
holds for all $\overline{v}$, with $|v|\leq k-d\leq
k-1$. Since these vectors span the Hilbert space $H$, this implies that both vectors
$\sum a_w\overline{x_iw}$ and $\sum b_w\overline{x_iw}$ are identical.
The definition (\ref{defxopt}) is also consistent with the involution on $\mathbb{K}\langle x\rangle$, i.e., it satisfies $\hat X^*_i=\hat X_{i+n}$. Indeed, for all $v,w$ of length $|v|,|w|\leq k-1$,
\be
\langle \overline{v},X^*_i\overline{w}\rangle=\langle X_i\overline{v}, \overline{w}\rangle=\langle \overline{x_iv},\overline{w}\rangle=y^k_{v^* x^*_iw}=y^k_{v^* x_{i+n}w}=\langle \overline{v},X_{i+n}\overline{w}\rangle\,.
\ee

Let us now, define  $\phi=\overline{1}$. Let $w$ be of length $|w|\leq 2k$ and write $w=w_1w_2$ with $|w_1|,|w_2|\leq k$. Then $\langle\phi,w(X)\phi\rangle=\langle\overline{w_1^*},\overline{w_2}\rangle=y^k_{w_1w_2}=y^k_w$. This implies that $\langle\phi,p(X)\phi\rangle=\sum_{|w|\leq 2k} p_w \langle\phi, w(X)\phi\rangle=\sum_{|w|\leq 2k} p_w y^k_w=p^k$. It remains to check that the operators $X$
satisfy $q_i({X})\succeq 0$. To verify this it is only necessary, because of (\ref{vectorial}), to show that the matrix $A$ with entries
$A(v,w)= \langle \overline{v},q_i({X})\overline{w}\rangle$
with $|v|,|w|\leq k-d$ is a positive semidefinite matrix. This is the case, since $A$ is equal to $M_{k-d}(q_i
y^{k})$, and is thus a submatrix of $M_{k-d_i}(q_i y^{k})\succeq 0$, which is itself positive semidefinite because $y^{k}$ is a solution of the
relaxation of order $k$.
\end{proof}

Note that there exists a related optimality detection criterion in the commutative case, which is based on the flat extension theorem of Curto and Fialkow \cite{curto, laurent}. The matrix $M_k(y^k)$ is said to be a \emph{flat extension} of $M_{k-d}(y^k)$ if $\mathrm{rank}\,M_k(y^{k})=\mathrm{rank}\,M_{k-d}(y^{k})$. When this condition holds, the flat extension theorem permits (in the commutative case) to extend the finite sequence $y^k$ to an infinite sequence $\hat y$ satisfying $\mathrm{rank}\,M_{k'}(\hat y)=\mathrm{rank}\,M_{k}(y^k)$ for all $k'\geq k$. The proof of Theorem~2 yields an NC analogue of this important result (simply define the infinite sequence $\hat y$ through $\hat y_w=\langle \phi, w(X)\phi\rangle$ where $\phi$ and $X$ are the vectors and operators defined in the proof of Theorem~2).

\subsection{Relation to the Positivstellensatz for non-commutative
polynomials} \label{proofpos}\
We now explain the link between the
convergence of the SDP relaxations and the Positivstellensatz for
non-commutative polynomials introduced by Helton and McCullough
\cite{heltonmc}. We proceed by analogy with the link that exists in
the commutative case between the convergence of Lasserre's
relaxations \cite{lasserre} and Putinar's Positivstellensatz
\cite{putinar}.

Consider the problem
\begin{equation}\label{dual2}
\begin{array}{cccll}
\lambda^{k}&=&\displaystyle\max_{\lambda,b_i,c_{ij}} &\multicolumn{1}{l}{\lambda}\\
&&\mathrm{s.t.}&p-\lambda= \sum_j
b^*_jb_j+\sum_{i=1}^{m}\sum_j
c^*_{ij}q_ic_{ij}\\
&&& \max_{j}\text{deg}(b_j)\leq k,\\
&&&\max_{j}\text{deg}(c_{ij})\leq k-d_i\,,
\end{array}
\end{equation}
where $b_j$ and $c_{ij}$ are polynomials.
The expression $\sum_i b^*_ib_i$ is known as a sum of squares (SOS) and the above problem is a polynomial SOS problem. As shown in Appendix B, this polynomial SOS problem can be formulated as an SDP problem, which turns out to be the dual of $\mathbf{R_k}$. This implies that the optimal solution of (\ref{dual2}) provides a lower bound on the solution of $\mathbf{R_k}$, i.e.,
\be\label{lklpk}
\lambda^{k}\leq p^{k}\,.
\ee
Alternatively, this last relation can be established  as follows. Let $\lambda$, $b_j$, $c_{ij}$ be a feasible solution of (\ref{dual2}) and $y$ be a feasible solution of (\ref{relax}). We show that $L_y(p-\lambda)=\sum_{w
}p_ w  y_ w  -\lambda\geq 0$, which implies (\ref{lklpk}).
As $L_y(p-\lambda)=\sum_jL_y(b_j^*b_j)+\sum_i\sum_jL_y(c_{ij}^*q_ic_{ij})$, it is sufficient to show that $L_y(b_j^*b_j)\geq 0$ and that $L_y(c_{ij}^*q_ic_{ij})\geq 0$. Writing $b_j=\sum_ w  b_{j,w }w $, we find
\begin{eqnarray}
L_y(b^*_jb_j) & = & \sum_{ v , w } b^*_{j,v }L_y(v^*w)b_{j, w }\nonumber\\
& = &\sum_{v,w }b^*_{j,v
}M_k(y)( v , w )b_{j, w }\geq 0\,,
\end{eqnarray}
where we have used the definition (\ref{mom}) of the moment matrix $M_k(y)$ in the second equality and the property that $M_k(y)\succeq 0$ to deduce the last inequality. Similarly,
\begin{eqnarray}
L_y(c^*_{ij}q_ic_{ij}) &= & \sum_{ v , w }c^*_{ij,v} \sum_ u  q_{i,u } L_y(v^*uw ) c_{ij,w}\nonumber\\
&= & \sum_{ v,w } c^*_{ij, v } M_k(q_iy)( v,w)c_{ij,w}\geq 0\,,
\end{eqnarray}
where we have used the definition (\ref{loc}) of the localizing matrix $M_k(q_iy)$ and the property $M_k(q_iy)\succeq 0$.

So far, we thus have that  $\lambda^{k}\leq p^{k}\leq p^\star$ for all $k$. We note now from the definition (\ref{polyncprog2}) that for any $\epsilon>0$, the polynomial $p(X)-\left(p^\star-\epsilon\right)$ is strictly positive on $\mathbf{S}_Q$. It then follows from the Positivstellensatz representation theorem of Helton and McCullough\footnote{The proof given by Helton and
McCullough only covers the case of polynomials with real
coefficients, but it is straightforward to generalize it to the
complex case, see for instance \cite{doherty}.} \cite{heltonmc} that
\be\label{soshc}
p-p^\star+\epsilon = \sum_jb_j^*b_j +\sum_i\sum_j c_{ij}^*q_ic_{ij}
\ee
for some polynomials $b_j$ and $c_{ij}$. Let $k\geq \max_{ij}\left\{\text{deg}(b_j),\text{deg}(c_{ij})+d_i\right\}$. Then $(\lambda, b_i, c_{ij})$ is a feasible solution of (\ref{dual2}) with objective value $p^\star-\epsilon$ and therefore $\lambda^{k}\geq p^\star-\epsilon$. It follows that $p^\star-\epsilon\leq
\lambda^{k}\leq p^{k}\leq p^\star$, which implies  $p^{k}\rightarrow p^\star$ since $\epsilon>0$ is arbitrary.

We thus have just shown that the convergence of the relaxations
$\mathbf{R_k}$ can be proved, alternatively to the proof given in
Subsection~\ref{csr}, using the Positivstellensatz for
non-commutative polynomials. In fact, both proofs are somewhat
equivalent and the proof presented in Subsection~\ref{csr} can
itself be viewed as an undirect proof of the Positivestellensatz for
non-commutative polynomials. The advantage of the proof given in
Subsection~\ref{csr} is that it is more constructive in spirit and it inspired the proof of Theorem~2 where a procedure is given to build an optimizer
$(H^\star,X^\star,\phi^\star)$. The proof that we have just given, on the other hand,
connects with the fascinating theory of positive polynomials. We see
for instance that an a priori bound on the maximal
degree $k$ necessary in the SOS decomposition (\ref{soshc}) would
yield information on the speed of convergence of the relaxations $\mathbf{R_k}$.

\subsection{Dealing with equality constraints}\label{dealeq}
The problem $\mathbf{P}$ can contain a set
of equality constraints $e_i(X)=0$ ($i=1,\ldots,m_e)$, which can be
enforced through the pairs of inequalities
$e_i(X)\succeq 0$ and $-e_i(X)\succeq 0$. Rather than writing down directly the corresponding relaxations $\mathbf{R}_k$, it can be advantageous to exploit these equalities to reduce the complexity of the problem.

The set of equalities
\be
E=\{e_i\mid i=1,\ldots,m_e\}\subseteq \mathbb{K}[x,x^*]
\ee
generates the ideal
\be I=\{\sum_{i} f_ie_ig_i\mid f_i,g_i\in \mathbb{K}[x,x^*]\}\,,\ee
which is such that any $p\in I$ satisfies $p(X)=0$ for operators $X$ such that $e_i(X)=0$ ($i=1,\ldots,m_e$). It is therefore sufficient to express every polynomial $p\in \mathbb{K}[x,x^*]$ modulo $I$, that is, to work in the quotient ring $\mathbb{K}[x,x^*]/I$. Let $B$ denote a monomial basis for $\mathbb{K}[x,x^*]/I$. Then we only need to consider polynomial expressions of the form $q=\sum_{ w \in B} q_ w  w $ since for every polynomial $p\in \mathbb{K}[x,x^*]$, there exists a unique $q=\sum_{w
\in B} q_ w  w $ such that $p-q\in I$. It is readily seen that all the results presented so far still hold when we work at relaxation step $k$ with the reduced monomial basis $B_k=B\cap W_{k}$. The relaxation $\mathbf{R_k}$ then corresponds to an optimization over the set variables $(y_w
)_{ w \in B_{2k}}$ and involves matrices $M_k(y)$ and $M_{k-d_i}(g_iy)$ of sizes $|B_{k}|\times |B_{k}|$ and $|B_{k-d_i}|\times |B_{k-d_i}|$, respectively. This represents a reduction in the complexity of the original problem.

All the problem of course consists in building a monomial basis $B$ for the quotient ring $\mathbb{K}[x,x^*]/I$. This can be done, e.g., if a finite Gr\"obner basis exists and can be computed efficiently for the ideal $I$ \cite{mora}. Here below we give two examples where such a reduced monomial basis $B$ is readily obtained.

\subsubsection{Hermitian variables}\label{hermsec}
Polynomials in hermitian variables are elements of the
free $*$-algebra $\mathbb{K}[ x]$ with generators
$x=(x_1,\ldots,x_n)$ and anti-involution $\ast$ defined on letters as $x_i^*=x_i$. Our previous results carry over to this situation if words are now viewed as built on the $n$ letters $x_1,\ldots,x_n$ rather than the $2n$ letters $x_1,\ldots,x_n,x_{n+1},\ldots,x_{2n}$ and if the anti-involution $*$ is re-interpreted accordingly.
 Since the algebra is now based on $n$ generators, the set of words of length $d$ has $|\mathcal{W}_d|=(n^{k+1}-1)/(n-1)$ elements, compared to
$((2n)^{k+1}-1)/(2n-1)$ for the general case in $2n$ variables. The size of the
optimization variables $y$ and the dimension of the moment and
localizing matrices in the SDP problem $\mathbf{R_k}$ are reduced accordingly.

\subsubsection{Commuting variables and link with Lasserre's results} The method that
we have presented to solve optimization problems in non-commuting
variables also contains, as a particular case, the commutative
version (\ref{polyprog}) considered by Lasserre since constraints of
the type $X_iX_j-X_jX_i=0$ can explicitly be imposed on the
operators $X_i$. More precisely,  the problem
\begin{equation}\label{polyprogcomm}
\begin{array}{llcll}
{p^{c}}&=&\displaystyle\min_{(H,X,\phi)}& \multicolumn{1}{l}{\langle \phi, p(X) \phi\rangle}\\
&&\text{s.t.}&q_i(X)\succeq
0&\quad i=1,\ldots,m \\
&&&X_iX_j-X_jX_i=0,&\quad i,j=1,\ldots,n\,,
\end{array}
\end{equation}
where the variables $X_i$ are assumed to be hermitian and all
polynomials are expressed in terms of real coefficients, is identical
to (\ref{polyprog}). To show that (\ref{polyprogcomm}) and
(\ref{polyprog}) are equivalent note that the operators $X$ in any
feasible solution $(H,X,\phi)$ of (\ref{polyprogcomm}) generate an
abelian algebra. Hence the Hilbert space $H$ (or at least the part
of $H$ on which the operators $X$ and the state $\phi$ have support)
is isomorphic to a direct integral $\int^{\oplus}H_x\,d\mu(x)$ of
one-dimensional Hilbert spaces $H_x$, and the operators $X_i$ are
decomposable as $X_i=\int^{\oplus} x_{i}\, d\mu(x)$, where each
$x_{i}$ is a scalar operator that acts only on $H_x$ \cite{neumann}.
A priori, any point $x\in \mathbb{R}^n$ defines a possible $n$-uple
of operators $(x_{1},\ldots,x_{n})$ and can be associated with
a factor $H_x$, but to satisfy (\ref{polyprogcomm}) the measure $d\mu(x)$ should be such that $\int_{S} d\mu(x)=1$ and $\int_{\mathbb{R}^n\setminus S} d\mu(x)=0$,
where $S=\{x\in\mathbb{R}^n\mid q_i(x)\geq 0,\; i=1,\ldots,m\}$.
Thus (\ref{polyprogcomm})  is equivalent to
\begin{equation}\label{polyprogmom}
\begin{array}{llcll}
p^c&=&\displaystyle\min_{\mu}& \multicolumn{1}{l}{\displaystyle\int p(x)d\mu(x)}\\
&&\text{s.t.}& \displaystyle\int_{S}
d\mu(x)=1,\;\int_{\mathbb{R}^n\setminus S} d\mu(x)=0\,,
\end{array}
\end{equation}
where the minimum is taken over all measures $\mu$ on
$\mathbb{R}^n$. As shown by Lasserre \cite{lasserre}, the problems
(\ref{polyprogmom}) and (\ref{polyprog}) are equivalent. Indeed, as
$p(x)\geq p^\star$ on $S$, $\int p d\mu\geq p^\star$ and thus
$p^c\geq p^\star$. On the other hand, if $x^\star$ is a global
minimizer of (\ref{polyprog}), then the measure
$\mu^\star=\delta_{x^\star}$ is admissible for (\ref{polyprogmom}),
and thus $p^c\leq p^\star$.

The relaxations $\mathbf{R_k}$ are constructed on the canonical basis
of non-commutative monomials, for instance for $n=2$,
$\mathcal{W}_2=\{1,x_1,x_2,x_1^2,x_1x_2,x_2x_1,x_2^2\}$. Simplifying these relaxations using the constraints
$x_ix_j-x_jx_i=0$ amounts to consider only the canonical basis of commutative
monomials, e.g.,
$\mathcal{W}^c_2=\{1,x_1,x_2,x_1^2,x_1x_2,x^2\}$, which lead to the exact same construction as the one introduced
by Lasserre.
In particular, the criterion for detecting global optimality presented in subsection~\ref{detectoptim} coincides with the
detection criterion introduced in the commutative situation
\cite{henrion2}. If we apply the procedure outlined in the proof of
Theorem~2 to extract optimal solutions from the solution of a finite order
relaxation $\mathbf{R_k}$, we end up with a set of operators $X=(X_1,\ldots,X_n)$ which are matrices each of dimension $r=\mathrm{rank}\,M_k(y^{k})$. As these matrices all commute, they can be simultaneously diagonalized, with each set of common eigenvalues
$(x_1(j),\ldots,x_n(j))$ $(j=1,\ldots,r)$ corresponding to one
optimal solution of (\ref{polyprog}). We thus see that if the rank of the moment matrix $r=\text{rank}\,M_k(y^{(k)})$ is related to the Hilbert space dimension of the global optimal solution in the
non-commutative case, it is related to the number of
global solutions extracted by the algorithm in the commutative case.

It is interesting to note that most of our results, such as the convergence of the hierarchy or the criterion to detect optimality, are easier to establish in the general non-commutative framework than they are in the specialized commutative case. Note also that it may be easier, from a computational point of view, to solve the non-commutative version of a problem than it is to solve the commutative one. In particular, the speed of convergence of the SDP relaxations may be faster in the non-commutative case than in the commutative one. This is dramatically illustrated on the following example.

Let $p$ be a polynomial of degree 2 and consider the quadratic problem
\begin{equation}\label{quadnc}
\begin{array}{llcll}
\quad p^\star&=&\displaystyle\min_{(H,X,\phi)}& \multicolumn{1}{l}{\langle \phi, p(X) \phi\rangle}\\
&&& X_i^2-X_i=0&\quad i=1,\ldots,n\,,
\end{array}
\end{equation}
where the variables $X_i$ are assumed to be hermitian.
Its first order relaxation is
\begin{equation}\label{relaxquadnc}
\begin{array}{ccllr}
\quad p^{1}&=&\displaystyle\min_{y} &\sum_\alpha
 p_ w  y_ w \\
&&\mathrm{s.t.} & y_1=1\\
&& & M_1(y)\succeq 0
\\ &&& y_{ii}-y_i=0 \quad i=1,\ldots,n\,.
\end{array}
\end{equation}
Any feasible point $y$ of the above SDP problem with objective value $p(y)=\sum_ w  p_ w  y_ w $ defines a feasible point of (\ref{quadnc}) with objective value $\langle \phi, p(X) \phi\rangle=p(y)$, and therefore $p^{1}=p^\star$, i.e., the first order relaxation already yields the global optimum of the original problem.
To see this, perform a Gram decomposition of the matrix $M_1(y)$: $M_1(y)(v,w )=y_{ vw}=\langle\overline{v},\overline{w} \rangle$, where $|v|,|w|\leq 1$, i.e., $v,w\in\{1,x_1,\ldots,x_n\}$. Define the vector $\phi=\overline{1}$, which is normalized since $\langle\overline{1},\overline{1}\rangle=y_1=1$, and the operator $X_i$ ($i=1,\ldots,n$) as the projectors on $\overline{x_i}$. Obviously, $X_i^2=X_i$. Moreover,
$X_i\phi=X_i\overline{x_i}+X_i(\phi-\overline{x_i})=\overline{x_i}$,
where the last equality follows from the fact that
$\langle\overline{x_i},\phi-\overline{x_i}\rangle=y_i-y_{ii}=0$. This implies that
$y_{vw}=\langle \phi,v(X) w(X)\phi\rangle$ for $v,w\in\{1,x_1,\ldots,x_n\}$ and therefore that $p(y)=\langle \phi, p(X) \phi\rangle$ since $p$ is of degree 2. Using similar arguments, one can actually show that the minimization of a polynomial of arbitrary degree evaluated over projection operators can always be determined from the first relaxation of the problem.

The commutative version of (\ref{quadnc}) is the quadratically constrained quadratic program
\begin{equation}\label{quadc}
\begin{array}{llclll}
p^\star &=&\displaystyle\min_{x\in\mathbb{R}^n}&p(x)\\
   & &\text{s.t.}& x_i^2-x_i=0&\qquad i=1,\ldots,n\,.
\end{array}
\end{equation}
Since 0-1 integer programming can be formulated in this form, it is NP-hard to solve a general instance of (\ref{quadc}). Thus, contrary to the non-commutative case, it is highly unlikely that considering relaxations up to some bounded order might be sufficient to solve this problem.

\subsection{Generalization}\label{gen}
In this subsection, we introduce a slight generalization of the problem (\ref{polyncprog2}) to which our method readily extends. We state the results without entering in the details of the proofs.

In addition to the polynomials $p$ and $\{q_i\,:\,i=1,\ldots,m_q\}$ defined in (\ref{polyncprog2}), consider the sets of polynomials $\{r_i\mid i=1,\ldots,m_r\}$ and
$\{s_i\mid i=1,\ldots,m_s\}$, where the $s_i$'s are hermitian. The problem that we consider is
\begin{equation}\label{polyncprog2b}
\mathbf{\tilde P}:\qquad\begin{array}{llcll}
\tilde{p}^{\star}&=&\displaystyle\min_{(H,X,\phi)}& \multicolumn{1}{l}{\langle \phi, p(X) \phi\rangle}\\
&&\text{s.t.}& q_i(X)\succeq
0&\quad i=1,\ldots,m_q\,,\\
&&& r_i(X)\phi=0&\quad i=1,\ldots,m_r\,,\\
&&&\langle \phi, s_i(X)\phi\rangle\geq 0&\quad i=1,\ldots,m_s\,.
\end{array}
\end{equation}
We thus not only require that the operators $X$ satisfy $q_i(X)\succeq 0$ but we also require that $r_i(X)$ acting on $\phi$ yield the null vector and that the average value of $s_i(X)$ be positive. As before we assume that $Q=\{q_i\,:\,i=1,\ldots,m_q\}$ is such that the quadratic module $\mathbf{M}_Q$ is Archimedean.

For $r\in\mathbb{K}[x,x^*]_d$  and
$y=\{y_w \}_{| w |\leq k+d}$, a sequence indexed in $\mathcal{W}_{k+d}$, define the vector $\mathrm{m}_k(r y)$ as the
vector with components indexed in $\mathcal{W}_k$ and whose component $ w $ is equal to
\begin{equation}
\mathrm{m}_k(ry)( w )=L_y(w  r)=\sum_{| v |\leq d} r_ v
y_{ wv }\,.
\end{equation}
If $y$ admits a moment representation (\ref{momrepr}) such that $r(X)\phi=0$, then
$\mathrm{m}_k(ry)=0$, since
\begin{eqnarray}
\mathrm{m}_k(fy)( w )&=&\sum_{ v }r_ v
y_{ w v }=\sum_ v  r_ v  \langle\phi,w(X)
v(X) \phi\rangle=\langle\phi,w(X)  r(X)\phi\rangle=0.
\end{eqnarray}
If in addition $y$ admits a moment representation such that
$\langle\phi,s(X)\phi\rangle\geq 0$, then obviously $\sum_{ w }s_ w
y_ w \geq 0$. These observations motivate the following definition.

For $2k\geq \max \left\{\mathrm{deg}(p),
\mathrm{deg}(q_i),\mathrm{deg}(r_i),\mathrm{deg}(s_i)\right\}]$, we define the relaxation of order $k$ associated to the problem $\mathbf{\tilde P}$ as the SDP problem
\begin{equation}\label{relaxb}
\mathbf{\tilde{R}_k}:\qquad\begin{array}{cclr@{\;}lr}
\tilde{p}^{k}&=&\displaystyle\min_{y} &\multicolumn{2}{c}{\sum_ w  p_ w  y_ w }\\
&&\mathrm{s.t.} & M_k(y)&\succeq 0\\
&& & y_1&=1\\
&& & M_{k-d_i}(q_iy)&\succeq 0& \quad i=1,\ldots,m_q\\
&&& \mathrm{m}_{2k-d'_i}(r_iy)&=0& \quad i=1,\ldots,m_r\\
 && & \sum_ w  s_{i, w } y_ w &\geq 0
 & \quad
i=1,\ldots,m_s\,,
\end{array}
\end{equation}
where $d_i=\lceil\text{deg}(q_i)/2\rceil$, $d'_i=\text{deg}(r_i)$,
and the optimization is over $y\in \mathbb{K}^{|\mathcal{W}_{2k}|}$.
It is easily verified that $\tilde{p}^{k}\geq \tilde{p}^{k)}$ when $k\leq k'$, and that $\tilde{p}^{k}\leq \tilde{p}^\star$ for all $k$.

The results obtained in Subsection~3.2 and 3.3 can easily be adapted to the above situation.
\begin{theorem}
If $\mathbf{M}_Q$ is Archimedean, $\lim_{k\rightarrow\infty} \tilde p^{k}=\tilde p^\star$.
\end{theorem}
\begin{theorem}
Assume that the optimal solution $y^{k}$ of the relaxation $\mathbf{\tilde{R}_k}$ of
order $k$ satisfies
\be
\mathrm{rank}\,M_k(y^{k})=\mathrm{rank}\,M_{k-d}(y^{k}),
\label{rankloopgen}\ee where
$d=\max_i d_i\geq 1$, and
\be\label{condd}
d'_i-d\leq k
\ee
for all $i=1,\ldots,m_r$. Then
$\tilde{p}^{k}=\tilde{p}^*$, i.e., the optimum of the relaxation of order $k$ is
the global optimum of the original problem $\mathbf{\tilde P}$.
Moreover, there exists a global optimizer
$(H^\star,X^\star,\phi^\star)$ of $\mathbf{\tilde P}$ with
$\dim{H^\star}= \mathrm{rank}\,M_{k-d}(y^{k})$.
\end{theorem}
The proof of both these theorems follow along the same line as the proofs of Theorem~1 and Theorem~2, respectively. One has simply to show that the reconstructed operators $\hat X$ and the state $\hat \phi$ satisfy the additional properties $r_i(\hat X)\hat\phi=0$ and $\langle \hat\phi, s_i(\hat X)\hat \phi\rangle\geq 0$. This can be established given the conditions $\mathrm{m}_{2k-d'_i}(r_iy)=0$ and $\sum_w
 s_{i,w
 } y_ w \geq 0$ present in $\mathbf{\tilde{R}_k}$. The additional constraint (\ref{condd}) with respect to Theorem~2 comes from the fact that to show that $r_i(X)\phi=0$, we need to show, because of
(\ref{vectorial}), that $\langle\overline{w},r_i(X)\phi\rangle=0$
for all $| w |\leq k-d$. This is implied by $\mathrm{m}_{2k-d'_i}(r_iy)=0$ when $2k-d'_i\geq k-d$, i.e., when (\ref{condd}) is satisfied.

The duals of the relaxations $\mathbf{\tilde{R}_k}$ can be shown to be equivalent to the problems
\begin{equation}\label{dual2gen}
\begin{array}{cccl}
\tilde{\lambda}^{k}&=&\displaystyle\max_{\lambda,b_i,c_{ij},f_i,g_i} &\multicolumn{1}{l}{\lambda}\\
&&\mathrm{s.t.}&p-\lambda= \sum_j
b^*_jb_j+\sum_{i=1}^{m_q}\sum_j
c^*_{ij}q_ic_{ij}\\
&&&\qquad\qquad+\sum_{i=1}^{m_r}\left(f_ir_i+r^*_if^*_i\right)+\sum_{i=1}^{m_s}g_is_i\\
&&& \max_{j}\text{deg}(b_j)\leq k,\\
&&&\max_{j}\text{deg}(c_{ij})\leq k-d_i,\\
&&&\text{deg}(f_{i})\leq 2k-d'_i,\\
&&&g_i\geq 0\,,
\end{array}
\end{equation}
where $b_j$, $c_{ij}$, $f_i$ are polynomials and $g_i$ are real numbers. From the decomposition of $p-\tilde{\lambda}^{k}$ appearing in (\ref{dual2gen}), it clearly follows that $p(X)-\tilde{\lambda}^{k}\geq 0$ for any $(H,X,\phi)$ satisfying the constraints in $\mathbf{\tilde P}$. Thus the solution of the dual (\ref{dual2gen}) provides a certificate that the optimal solution $\tilde{p}^\star$ of $\mathbf{\tilde P}$ cannot be lower than $\tilde{\lambda}^{k}$.

Finally, we mention that it is possible, taking inspiration from
\cite{henrion}, to generalize the problem (\ref{polyncprog2}) and the results associated to it to the case of matrix-valued polynomials, that is, polynomials $\sum_w
 p_w
  w $, where each coefficient $p_ w $ is now an
$a\times b$ matrix with entries from $\mathbb{K}$. A Positivstellensatz also exists in this case \cite{heltonmc}.

\section{Illustration of the method}
For the sake of illustration, we now apply our approach
on simple examples. To simplify the notation, through all this section we label monomials (i.e. words) by the indices of the ordered non-commutative variables of which they are composed. For instance, the word $w=x_2x_1x_2x_2$ will be referred to as $2122$. The empty word $1$ corresponding to the identity element of the algebra will be labeled by the symbol $\emptyset$.

Our first example involves two hermitian variables
$X_1=X_1^*$ and $X_2=X_2^*$ and has the form
\begin{equation}\label{illustr}
\begin{array}{llcl}
p^\star&=&\displaystyle\min_{(H,X,\phi)}& \multicolumn{1}{l}{\langle \phi, X_1X_2+X_2X_1 \phi\rangle}\\
&&\text{s.t.}& X_1^2-X_1=0\\
&&& -X_2^2+X_2+1/2\succeq 0\,.\\
\end{array}
\end{equation}
Since all constraint and objective variables are at most of degree
2, the first order relaxations is associated with the monomial basis
$\mathcal{W}_2=\{1,x_1,x_2,x_1x_2,x_2x_1,x_2^2\}$, where, following the approach of Subsection~3.5, we used that $x_1^2=x_1$. The first relaxation step thus involves the relaxed variables
$\{y_\emptyset,y_1,y_2,y_{12},y_{21},y_{22}\}$ and corresponds to the SDP problem
\begin{equation}\label{illustrrel1}
\begin{array}{ccll}
p^{1}&=&\displaystyle\min_{y} &\multicolumn{1}{l}{y_{12}+y_{21}}\\
&&\mathrm{s.t.}  &\left[
\begin{array}{c|cc}
1 &y_1& y_2\\
\hline y_1& y_1& y_{12}\\
y_2 &y_{21}& y_{22}
\end{array}\right]\succeq 0\\
&& & -y_{22}+y_2+1/2\geq 0\,.\\
\end{array}
\end{equation}
We solved this SDP problem using the Matlab toolboxes YALMIP
\cite{yalmip} and SeDuMi \cite{sedumi}. After rounding, we obtain
the solution $p^{1}=-3/4$, achieved for the moment matrix
\be \label{m1}M_1=\left[\begin{array}{c|cc}
1 & 3/4 & -1/4\\
\hline 3/4 & 3/4 & -3/8\\
-1/4 & -3/8&1/4\end{array}\right]\,,
\ee
with eigenvalues $0$, $1\pm\sqrt{37}/8$. The second order relaxation
is
\begin{equation}\label{illustrrel2}
\begin{array}{ccll}
p^{2}&=&\displaystyle\min_{y} &\multicolumn{1}{l}{y_{12}+y_{21}}\\
&&\mathrm{s.t.}  &\left[
\begin{array}{c|cc|ccc}
1 &y_1& y_2 &y_{12}&y_{21}&y_{22}\\
\hline y_1& y_1& y_{12}&y_{12}&y_{121}&y_{122}\\
y_2 &y_{21}& y_{22}&y_{212}&y_{221}&y_{222}\\
\hline y_{21} &y_{21}& y_{212}&y_{212}&y_{2121}&y_{2122}\\
y_{12} &y_{121}& y_{122}&y_{1212}&y_{1221}&y_{1222}\\
y_{22} &y_{221}& y_{222}&y_{2212}&y_{2221}&y_{2222}\\
\end{array}\right]\succeq 0\\
&& &\left[\begin{array}{c|cc}
-y_{22}+y_2+\frac{1}{2} &-y_{221}+y_{21}+\frac{1}{2}y_{1}&-y_{222}+y_{22}+\frac{1}{2}y_{2} \\
\hline -y_{221}+y_{21}+\frac{1}{2}y_{1}& -y_{1221}+y_{121}+\frac{1}{2}y_{1}& -y_{1222}+y_{122}+\frac{1}{2}y_{12}\\
-y_{222}+y_{22}+\frac{1}{2}y_{2}
&-y_{1222}+y_{122}+\frac{1}{2}y_{12}&
-y_{2222}+y_{222}+\frac{1}{2}y_{22}
\end{array}\right]\succeq 0\,,\end{array}
\end{equation}
with solution $p^{2}=-3/4$. The moment matrix associated to this
solution is
\be M_2=\left[\begin{array}{c|cc|ccc}
1 & 3/4 & -1/4 &-3/8&-3/8&1/4\\
\hline 3/4 & 3/4 & -3/8 &-3/8&-3/16&0\\
-1/4 & -3/8 & 1/4
&3/16&0&1/8\\
\hline -3/8 & -3/8 & 3/16 &3/16&3/32&0\\
-3/8 & -3/16& 0 &3/32&3/16&-3/16\\
1/4&0&1/8&0&-3/16&1/4
\end{array}\right]\,,
\ee
which as two non-zero eigenvalues $3/32\times\left(14\pm\sqrt{61}\right)$.

\paragraph{Optimality criterion and extraction of optimizers.}
Since the matrix $M_2$ has two non-zero eigenvalues, it has rank $2$.
Let $M_1(y^{2})$ be the upper-right $3\times 3$ submatrix of $M_2=M_2(y^{2})$. This submatrix is, in fact, equal to (\ref{m1}) and has thus also rank $2$. The matrices $M_1(y^{2})$ and $M_2(y^{2})$ have thus the same rank and the condition (\ref{rankloop}) of Theorem~2 is satisfied. It follows that
$p^\star=p^{2}=-3/4$. It also follows that we can extract a global optimizer for (\ref{illustrrel1}), which will be realized in a space of dimension 2.  For this, write down the Gram
decomposition $M_2=R^TR$, where
\be
R=\left[\begin{array}{cccccc} 1 & 3/4 & -1/4 &-3/8 &
-3/8& 1/4\\
0& \sqrt{3}/4 & -\sqrt{3}/4 & -\sqrt{3}/8 &\sqrt{3}/8 & -\sqrt{3}/4
\end{array}\right]\,.
\ee
Following the procedure specified in the proof of Theorem~2, we
find the optimal solutions
\be\label{solillustr}
X_1^\star=\left[\begin{array}{cc} 3/4 & \sqrt{3}/4\\
\sqrt{3}/4&1/4 \end{array}\right]\,,\qquad
X_2^\star=\left[\begin{array}{cc} -1/4 & -\sqrt{3}/4\\
-\sqrt{3}/4&5/4 \end{array}\right]\,,\quad
\phi^\star=\left[\begin{array}{c}1\\0\end{array}\right]\,.
\ee

\paragraph{Dual.}
Solving the dual of the order 1 relaxation (\ref{illustrrel1}) yields, in the notation of Appendix~B,  the solutions
\begin{eqnarray}
\lambda &=& -3/4\nonumber\\
V&=&\left[\begin{array}{ccc}
1/4 & -1/2& -1/2\\
-1/2& 1 & 1 \\
-1/2& 1& 1
\end{array}\right]\nonumber\\
W&=&1\,.
\end{eqnarray}
The matrix $V$ has only one non-zero eigenvalue and can be written as $V=aa^T$ where $a=[-1/2,1,1]$. In the formlation of (\ref{dual2}), this corresponds to an SOS decomposition for $x_1x_2+x_2x_1$ of the form
\be\label{sos}
x_1x_2+x_2x_1-\left(-\frac{3}{4}\right)=\left(-\frac{1}{2}+x_1+x_2\right)^2+\left(-x_2^2+x_2+\frac{1}{2}\right)\,.
\ee
It immediately follows that
$\langle\phi,X_1X_2+X_2X_1\phi\rangle\geq -3/4$ for every
$(H,X,\phi)$ satisfying $X_1^2=X_1$ and $-X_2^2+X_2+\frac{1}{2}\succeq 0$. Thus the
decomposition (\ref{sos}) provides a certificate that the solution
(\ref{solillustr}) is optimal.

\paragraph{Comparison with the commutative case.}
To illustrate the differences and similarities between the
non-commutative and commutative case, let
\begin{equation}\label{illustr2comm}
\begin{array}{llcl}
p^\star&=&\displaystyle\min_{x\in\mathbb{R}^2}& \multicolumn{1}{l}{2x_1x_2}\\
&&\text{s.t.}& x_1^2-x_1=0\\
&&& -x_2^2+x_2+1/2\geq 0
\end{array}
\end{equation}
be the commutative version of (\ref{illustr}).
The first relaxation step associated to this problem involves  the monomial basis
$\mathcal{W}^c_2=\{1,x_1,x_2,x_1x_2,x_2^2\}$ (we used $x_1x_2=x_2x_1$) and the corresponding
relaxation variables
$\{y_\emptyset,y_1,y_2,y_{12},y_{22}\}$,. This should be compared to
$\mathcal{W}_2=\{1,x_1,x_2,x_1x_2,x_2x_1x_2^2\}$ and
$\{y_\emptyset,y_1,y_2,y_{12},y_{21},y_{22}\}$ in the
non-commutative case. The first order relaxation associated to (\ref{illustr2comm}) is thus
\begin{equation}\label{illustrrel1comm}
\begin{array}{ccll}
p^{1}&=&\displaystyle\min_{y} &\multicolumn{1}{l}{2\,y_{12}}\\
&&\mathrm{s.t.}  &\left[
\begin{array}{c|cc}
1 &y_1& y_2\\
\hline y_1& y_1& y_{12}\\
y_2 &y_{12}& y_{22}
\end{array}\right]\succeq 0\\
&& & -y_{22}+y_2+1/2\geq 0\,.\\
\end{array}
\end{equation}
Note that (\ref{illustrrel1}) and (\ref{illustrrel1comm}) are in
fact identical, because the hermicity of the moment matrix in
(\ref{illustrrel1}) implies that $y_{12}=y_{21}$. In general, it
always happen that the first order relaxations of the commutative and non-commutative version of a problem coincide. We thus
find as before that $p^{1}=-3/4$. The relaxation of order two of
(\ref{illustrrel1comm}), however, is
\begin{equation}\label{illustrrel2comm}
\begin{array}{ccll}
p^{2}&=&\displaystyle\min_{y} &\multicolumn{1}{l}{2y_{12}}\\
&&\mathrm{s.t.}  &\left[
\begin{array}{c|cc|cc}
1 &y_1& y_2 &y_{12}&y_{22}\\
\hline y_1& y_1& y_{12}&y_{12}&y_{122}\\
y_2 &y_{12}& y_{22}&y_{122}&y_{222}\\
\hline y_{12} &y_{12}& y_{122}&y_{122}&y_{1222}\\
y_{22} &y_{122}& y_{222}&y_{1222}&y_{2222}\\
\end{array}\right]\succeq 0\\
&& &\left[\begin{array}{c|cc}
-y_{22}+y_2+\frac{1}{2} &-y_{122}+y_{12}+\frac{1}{2}y_{1}&-y_{222}+y_{22}+\frac{1}{2}y_{2} \\
\hline -y_{122}+y_{12}+\frac{1}{2}y_{1}& -y_{122}+y_{12}+\frac{1}{2}y_{1}& -y_{1222}+y_{122}+\frac{1}{2}y_{12}\\
-y_{222}+y_{22}+\frac{1}{2}y_{2} &-y_{1222}+y_{122}+\frac{1}{2}y_{12}&
-y_{2222}+y_{222}+\frac{1}{2}y_{22}
\end{array}\right]\succeq 0\,.
\end{array}
\end{equation}
Solving it, we obtain $p^{2}=1-\sqrt{3}\simeq -0.7321$. Again, it can be verified that the
rank condition (\ref{rankloop}) of Theorem~2 is satisfied, so that this solution is optimal, and the following optimizer can be reconstructed:
\be
x_1^\star=1,\qquad x_2^\star=(1-\sqrt{3})/2\,.
\ee
As expected, the global minimum of (\ref{illustr2comm}) is higher
than the one of (\ref{illustr}) as the commutative case is more
constrained than the non-commutative one.

\paragraph{Additional constraints.} We now consider a problem of the form (\ref{polyncprog2b}) by adding two constraints to (\ref{illustr}):
{\allowdisplaybreaks
\begin{equation}\label{illustrgen}
\begin{array}{llcl}
p^\star&=&\displaystyle\min_{(H,X,\phi)}& \multicolumn{1}{l}{\langle \phi, X_1X_2+X_2X_1 \phi\rangle}\\
&&\text{s.t.}& X_1^2-X_1=0\\
&&& -X_2^2+X_2+1/2\succeq 0\\
&&& (3X_1+2X_2-1)\,\phi = 0\\
 &&&-\langle \phi, X_1\phi\rangle+1/3\geq 0\,.
\end{array}
\end{equation}}
Following (\ref{relaxb}), the corresponding first order relaxation is
{\allowdisplaybreaks\begin{equation}\label{illustrrel1gen}
\begin{array}{ccll}
p^{1}&=&\displaystyle\min_{y} &\multicolumn{1}{l}{y_{12}+y_{21}}\\
&&\mathrm{s.t.}  &\left[
\begin{array}{c|cc}
1 &y_1& y_2\\
\hline y_1& y_1& y_{12}\\
y_2 &y_{21}& y_{22}
\end{array}\right]\succeq 0\\
&& & -y_{22}+y_2+1/2\geq 0\\
& & & 3y_{\alpha
}+2y_{ v }-y_{ u }=0\qquad ( w , v , u )\in J_1\\
&& & -y_1+1/3\geq 0\,,
\end{array}
\end{equation}}
where $J_1=\{(1,2,\emptyset),(1,12,1),(2,22,2)\}$.
This problem admits the solution $p^{1}=-2/3$, achieved for the moment matrix
\be M_1=\left[\begin{array}{c|cc}
1 & 1/3 & 0\\
\hline 1/3 & 1/3 & -1/3\\
0 & -1/3 & 1/2\end{array}\right]\,,
\ee
with eigenvalues $0$, $2/3$, and $7/6$. The solution $p^{1}=-2/3$ thus yields a lower-bound on $p^\star$, which is already higher, as expected, than the optimal solution of (\ref{illustr}).  The second order relaxation
is
{\allowdisplaybreaks
\begin{align}\label{illustrrel2gen}
p^{2}=&\displaystyle\min_{y}\; {y_{12}+y_{21}}\nonumber\\
&\mathrm{s.t.}\;  \left[
\begin{array}{c|cc|ccc}
1 &y_1& y_2 &y_{12}&y_{21}&y_{22}\\
\hline y_1& y_1& y_{12}&y_{12}&y_{121}&y_{122}\\
y_2 &y_{21}& y_{22}&y_{212}&y_{221}&y_{222}\\
\hline y_{21} &y_{21}& y_{212}&y_{212}&y_{2121}&y_{2122}\\
y_{12} &y_{121}& y_{122}&y_{1212}&y_{1221}&y_{1222}\\
y_{22} &y_{221}& y_{222}&y_{2212}&y_{2221}&y_{2222}\\
\end{array}\right]\succeq 0\\
&\phantom{\mathrm{s.t.}}\;\left[\begin{array}{c|cc}
-y_{22}+y_2+\frac{1}{2} &-y_{221}+y_{21}+\frac{1}{2}y_{1}&-y_{222}+y_{22}+\frac{1}{2}y_{2} \\
\hline -y_{221}+y_{21}+\frac{1}{2}y_{1}& -y_{1221}+y_{121}+\frac{1}{2}y_{1}& -y_{1222}+y_{122}+\frac{1}{2}y_{12}\\
-y_{222}+y_{22}+\frac{1}{2}y_{2}
&-y_{1222}+y_{122}+\frac{1}{2}y_{12}&
-y_{2222}+y_{222}+\frac{1}{2}y_{22}
\end{array}\right]\succeq 0\nonumber\\
&\phantom{\mathrm{s.t.}}\; 3y_{ w }+2y_{ v }-y_{ u }=0\qquad ( w , v , u )\in J_2\nonumber\\
 &\phantom{\mathrm{s.t.}}\; -y_1+1/3\geq 0\nonumber\,,
\end{align}}
where $J_2=\{(1,2,\emptyset),(1,12,1),(21,22,2),(121,122,12),(21,22,21),(221,222,22),\linebreak[4](121,1212,121),(1221,1222,122),(2121,2122,212),(221,2212,221),(2221,2222,222)\}$. It admits the solution $p^{2}=-2/3$ with
\be M_2=\left[\begin{array}{c|cc|ccc}
1 & 1/3 & 0 &-1/3&-1/3&1/2\\
\hline 1/3 & 1/3 & -1/3 &-1/3&0&-1/6\\
0 & -1/3 & 1/2
&1/3&-1/6&1/2\\
\hline -1/3 & -1/3 & 1/3 &1/3&0&1/6\\
-1/3 & 0 & -1/6 &-0&1/6&-1/3\\
1/2&-1/6&1/2&1/6&-1/3&3/4
\end{array}\right]\,,
\ee
which as two non-zero eigenvalues $17/12$ and $5/3$. As in the previous examples, it is easily verified that the rank
condition (\ref{rankloop}) is satisfied, and we thus deduce that
$p^\star=p^{2}=-2/3$. From the Gram
decomposition $M_2=R^TR$, with
\be
R=\left[\begin{array}{cccccc} 1 & 1/3 & 0 &-1/3 &
-1/3& 1/2\\
0& \sqrt{2}/3 & -\sqrt{2}/2 & -\sqrt{2}/3 &\sqrt{2}/6 & -\sqrt{2}/2
\end{array}\right]\,,
\ee
one obtains the global optimizer
\be\label{solillustr2}
X_1^\star=\left[\begin{array}{cc} 1/3 & \sqrt{2}/3\\
\sqrt{2}/3&2/3 \end{array}\right]\,,\qquad
X_2^\star=\left[\begin{array}{cc} 0 & -\sqrt{2}/2\\
-\sqrt{2}/2&1 \end{array}\right]\,,\quad
\phi=\left[\begin{array}{c}1\\0\end{array}\right]\,.
\ee
Finally, the dual of the first order relaxation (\ref{illustrrel1gen}) yields the SOS decomposition:
\begin{eqnarray}\label{sos2}
x_1x_2+x_2x_1-\left(-\frac{2}{3}\right)&=&\frac{1}{9}\left(-1+3x_1+2x_2\right)^2+\frac{4}{9}\left(-x_2^2+x_2+\frac{1}{2}\right)+\left(\frac{1}{3}-x_1\right)\nonumber\\
&&+\frac{1}{6}x_1\left(3x_1+2x_2-1\right)+\frac{1}{6}\left(3x_1+2x_2-1\right)x_1\,,
\end{eqnarray}
which clearly implies $p^\star\geq -2/3$.

\section{Applications}
The results presented so far have immediate applications in quantum theory and quantum information science.
Since the dimension of the underlying Hilbert space is not specified in the optimization problem (\ref{polyncprog2}) or (\ref{polyncprog2b}), they are well adapted to situations where we want to optimize a quantity over all its possible physical realizations, that is to say, over Hilbert spaces of arbitrary dimension. Computing the maximal quantum violation of a Bell inequality is an example of this sort.

Let $S_1,\ldots,S_N$ be a collection of finite disjoint sets. Each $S_k$ represents a measurement that can be performed on a given system and the elements $i\in S_k$ are the possible outcomes of the $k$-measurement. We suppose that the system is composed of two non-interacting subsystems, and that measurements $S_1,\ldots,S_n$ are performed on the first system and measurements $S_{n+1},\ldots,S_N$ on the second. We put $A=S_1\cup\ldots\cup S_n$, $B=S_{n+1}\cup\ldots\cup S_N$, and denote by $P(ij)$ the joint probability to obtain outcome $i\in A$ and outcome $j\in B$ when measurements associated to these outcomes are made on the first and second subsystems, respectively. In quantum theory, these probabilities are given by $P(ij)=\langle\phi,E_iE_j\phi\rangle$, where $\phi$ describes the state of the system under observation and the self-adjoint operators $E_i$ describe the measurements performed on $\phi$. The measurement operators $\{E_i\,:\,i\in S_k\}$ associated to the measurement $S_k$ form an orthogonal resolution of the identity, and operators corresponding to different subsystems commute, i.e., $[E_i,E_j]=0$ when $i\in A$ and $j\in B$.

For our purposes, a Bell inequality is simply a linear expression $\sum_{ij}c_{ij}P(ij)$ in the joint probabilities. We are interested in the maximal value that this quantity can take over all probabilities $P(ij)$ that admit a quantum representation. This amounts to solve the problem
\begin{align}\label{bell}
\displaystyle\min_{(H,E,\phi)}\quad& \langle\phi,\sum_{ij}c_{ij}E_iE_j\phi\rangle\nonumber\\
\text{s.t.}\quad& E_iE_j=\delta_{ij}E_i\quad \forall S_k \text{ and } \forall i,j\in S_k\nonumber\\
& \sum_{i\in S_k} E_i=1 \qquad\forall S_k\\
&[E_i,E_j]=0 \qquad\forall i\in A \text{ and } \forall j\in B\nonumber\,,
\end{align}
which is a particular instance of the non-commutative optimization problem (\ref{polyncprog2}) and involves polynomials of degree at most 2. Note that $1-\sum_{i\in\mathcal{S}_k}E^2_i=(1-\sum_{i\in \mathcal{S}_k} E_i)+\sum_{i\in\mathcal{S}_k}(E_i-E^2_i)=0$, and thus the quadratic module associated to the constraints in (\ref{bell}) is Archimedean.  The sequence of SDP relaxations associated to (\ref{bell}) thus converges to the optimal solution. This particular sequence of SDP relaxations is the one already introduced in \cite{mpa,mpa2} and the source of inspiration for the present work. It represents the unique tool that is currently available to compute the maximal violation of a generic Bell inequality. It has been applied up to the third order in \cite{vertesi} to derive upper-bounds on the maximal violation of 241 Bell inequalities. The resulting upper-bounds are tight for all but 20 of these inequalities; for the remaining 20 inequalities the gap between the upper bound and the best known lower bound is small.

The sequence of SDP relaxations introduced here can also be used to decide if a given set of probabilities $P(ij)$ admits a quantum representation \cite{mpa,mpa2}. More generally, it has the potential to find other applications in quantum information science, see for instance \cite{doherty,ito,randomness}.

Besides applications where the dimension of the underlying Hilbert space is not fixed, the optimization problems (\ref{polyncprog2}) and (\ref{polyncprog2b}) are also well suited to problems where the Hilbert space is the unique irreducible representation space of a set of operators satisfying algebraic constraints. Consider, for instance a system of $N$ electrons that can occupy $M$ orbitals, each orbital being associated with annihilation and creation operators $a_i$ and $a^\dagger_i$, $i=1,\dots,M$ (we use the common physics notation $\dagger$ for the conjugate transpose). Since electrons interact pairwise, the hamiltonian for such a system involve only two-body interactions and its ground state energy can be computed as
{\allowdisplaybreaks\begin{align}\label{qc}
\min\quad& \langle\phi,\,\sum_{ijkl}h_{ijkl}a_i^\dagger a_j^\dagger a_k a_l\,\phi\rangle\nonumber\\
\text{s.t.}\quad& \{a_i,a_j\}=0\nonumber\\
& \{a_i^\dagger,a_j^\dagger\}=0\\
& \{a_i^\dagger,a_j\}=\delta_{ij}\nonumber\\
&\left(\sum_{i} a^\dagger_i a_i - N\right)\phi=0\,.\nonumber
\end{align}}
The first three constraints represent the usual anticommutation fermionic relations, while the last constraint fixes the number of electrons to $N$. This problem is a particular case of (\ref{polyncprog2b}) and it involves polynomials of degree 4. Note that the algebra of operators generated by (\ref{qc}) has a unique irreducible representation of dimension $2^M$. Since a product (in normal order) of more than $N$ of the operators $\{a_i,a^\dagger_i\}$ vanishes, the sequence of SDP relaxations halts at order $N$, and thus $p^{N}=p^\star$.

The hierarchy of SDP relaxations associated to the problem (\ref{qc}) can be used, for instance, to compute the ground state electronic energy of atoms or molecules. In the last years, very successful SDP methods based on the $N$-representability problem have been independently introduced in quantum chemistry to compute these electronic energies \cite{mazziotti,mazziotti2}. Our hierarchy of SDP relaxations actually reduces to these existing SDP techniques. But our approach is more general, and can be used to compute the ground-state energy of other many-body systems, such as spin systems or systems described by \emph{unbounded} operators satisfying the canonical relations $[x,p]=i$ (in which case it has to be slightly adapted). These applications will be presented in a forthcoming paper.

Finally, the method presented here might also prove useful for problems where the Hilbert space dimension is fixed in advance.
Consider for instance a polynomial optimization problem of the form (\ref{polyncprog2}) where \linebreak[1] $\text{dim } H=r$, i.e, where the operators $X$ are $r\times r$ matrices. We may in principle solve such a
problem by introducing an explicit parametrization of the matrices $X$ and by using Lasserre's method for polynomial scalar optimization \cite{lasserre} or its extension taking into account polynomial
matrix inequalities \cite{henrion}. This would necessitate, however, to introduce of the order of $r^2$ scalar variables for each operator $X_i$. This renders this approach impractical even for small problems.
In comparison, the method presented here treats each matrix as a single variable. Although it only represents a relaxation of the original problem since the Hilbert space dimension is not fixed (in particular we have no guarantee that the sequence of relaxations will converge to a solution with $\text{dim }H=r$), it may nevertheless provide a
cheap way to compute lower-bound on the optimal solutions of these problems when it is too costly to introduce an explicit parametrization.

\section{Acknowledgements}
We are grateful to Jean Bernard Lasserre and Mihai Putinar for helpful discussions. We thank Ben Toner for pointing out to us reference~\cite{heltonmc} and anonymous referees for their constructive comments.
S.P acknowledges support by the Swiss NCCR Quantum Photonics and the EU Integrated Project QAP. M.N. acknowledges support from an Institute for Mathematical Sciences Fellowship. We thank the European QAP and PERCENT projects, the Spanish MEC
FIS2007-60182 and Consolider-Ingenio QOIT projects, and the Generalitat de
Catalunya and Caixa Manresa for financial support.

\appendix
\section*{Appendix A: Basics of semidefinite programming}\label{appsdp}
\addcontentsline{toc}{section}{Appendix A: Basics of semidefinite programming}

Semidefinite programming \cite{boyd-vande} is a subfield of convex
optimization concerned with the following optimization problem,
known as the \emph{primal problem}
\begin{eqnarray}
\mbox{minimize}\quad && c^Tx\nonumber\\
\mbox{subject to}\quad &&F(x)=\sum_{i=1}^m x_iF_i-G\succeq 0\,.
\label{primal}
\end{eqnarray}
The problem variable is the vector $x$ with $m$ components $x_i$ and the problem
parameters are the $n\times n$ matrices $G,F_i$ and the scalars $c_i$.  A vector $x$ is said to be
\emph{primal feasible} when $F(x)\geq 0$.

For each primal problem there is an associated \emph{dual
problem}, which is a maximization problem of the form
\begin{eqnarray}
\mbox{maximize}\quad && \mbox{tr}(GZ)\nonumber\\
\mbox{subject to}\quad &&\mbox{tr}F_iZ=c_i \quad i=1,...,m\label{dualb}\\
&& Z\succeq 0\nonumber
\end{eqnarray}
where the optimization variable is the $n\times n$ matrix $Z$.
The dual problem is also a semidefinite program, i.e., it can be
put in the same form as (\ref{dualb}).
A matrix $Z$ is said to be \emph{dual feasible} if it
satisfies the conditions in (\ref{dualb}).

The key property of the dual program is that it yields bounds on the
optimal value of the primal program. To see this, take a primal
feasible point $x$ and a dual feasible point $Z$. Then $c^T
x-\mbox{tr}(GZ) = \sum_{i=1}^m\mbox{tr}(ZF_i)x_i - \mbox{tr}(GZ) =
\mbox{tr}(ZF(x))\geq 0$. This proves that the optimal primal value
$p^*$ and the optimal dual value $d^*$ satisfy $d^*\leq p^*$. In
fact, it usually happens that $d^*=p^*$. A sufficient condition for
this to hold is that the dual (primal) problem admits a strict feasible point, that is, that there exists a matrix
$Z\succ 0$ ($F(x)\succ 0$)) that is dual (primal) feasible \cite{boyd-vande}. We refer the reader to the review of Vandenberghe and Boyd \cite{boyd-vande} for further information on SDP.

There exist many available numerical packages to solve SDPs, for
instance for Matlab, the toolboxes SeDuMi \cite{sedumi} and YALMIP
\cite{yalmip}. These algorithms solve both the primal and
the dual at the same time and thus yields bounds on the accuracy
of the solution that is obtained.

\section*{Appendix B: Duals of the SDP relaxations}
\addcontentsline{toc}{section}{Appendix B: Dual of the SDP relaxations}

Here we show that the duals of the relaxations $\mathbf{R_k}$ defined in (\ref{relax}) correspond to the problems (\ref{dual2}). To simplify the presentation, we do
this explicitly only in the case where we are dealing with
polynomials defined in the real free $*-$algebra $\mathbb{R}[x,x^*]$ and where the SDP relaxations (\ref{relax}) only
involves real quantities. The more general case of complex SDP
relaxations can be treated similarly by decomposing them in real and
imaginary parts.

Write
$M_k(y)=\sum_w
 B_ w  y_ w $ and
$M_{k-d_i}(q_iy)=\sum_ w  C^i_{ w } y_ w $ for appropriate
symmetric matrices $B_ w $ and $C^i_ w $.
The SDP relaxation (\ref{relax}) is
then expressed as an SDP problem in primal form (\ref{primal}) and its dual
is
\begin{equation}\label{dual}
\begin{array}{cccl}
\lambda^{k}&=&\displaystyle\max_{\lambda,V,W_i} &\multicolumn{1}{l}{\lambda}\\
&&\mathrm{s.t.}& p_1=\lambda+\text{tr}\left(B_\emptyset
V\right)+\sum_{i=1}^{m}\text{tr}(C^i_\emptyset
W_i)\\
&&& p_ w =\text{tr}\left(B_ w
V\right)+\sum_{i=1}^{m}\text{tr}(C^i_ w
W_i)
\qquad \qquad(\forall\, 0<| w |\leq 2k)\\
&&& V\succeq 0,\\
&&& W_i\succeq 0, \quad i=1,\ldots,m\,, \end{array}
\end{equation}
where $\lambda\in\mathbb{R}$,
$V\in\mathbb{R}^{|\mathcal{W}_k|}\times\mathbb{R}^{|\mathcal{W}_k|}$, and
$W_i\in\mathbb{R}^{|\mathcal{W}_{k-d_i}|}\times\mathbb{R}^{|\mathcal{W}_{k-d_i}|}$.

The terms on the left hand-sides of the above equality constraints
are the coefficients in the canonical basis of monomials
$\mathcal{W}_{2k}=\{w \,:\,| w |\leq 2k\}$ of the polynomial
$p$. The quantities $\text{tr}(B_ w  V)$ on the right-hand side are the
coefficients of a polynomial of the form $\sum_j b^*_jb_j$, where each $b_j$ is a polynomial of degree $k$.
Indeed, it is easily seen from the definition of the moment matrix
$M_k(y)$ that the entries of the matrices $B_ w $ satisfy
$B_ w ( u,v )=1$ if $ w = u ^*v $ or $ w = v ^*u $ and
$B_w
( u,v )=0$ otherwise. It follows that
$\sum_{| w |\leq 2k} \text{tr}(B_ w  V)
w =\sum_{| u |,| v |\leq k} V_{ v  u }\,u^*
v $, where we used that $V$ is symmetric. As $V$ is positive semidefinite, we can write $V=\sum_j
\mu_j a_j a_j^T$, where $\mu_j\geq 0$ are the eigenvalues of
$V$ and $a_j$ the corresponding eigenvectors. Using this expression
for $V$, we obtain that $\sum_{|w|\leq 2k} \text{tr}(B_ w
V) w  = \sum_j \mu_j a^*_ja_j$, which is of the
announced form with $b_j=\sqrt{\mu_i}a_j$. In a similar way,
it can be shown that $\sum_{| w |\leq 2k} \text{Tr}(C^i_w
W_i) w  = \sum_j c^*_{ij}q_ic_{ij}$.  Putting all together, we find that
the the problem (\ref{dual}) is equivalent to
\begin{equation}\label{dual3}
\begin{array}{cccll}
\lambda^{k}&=&\displaystyle\max_{\lambda,b_i,c_{ij}} &\multicolumn{1}{l}{\lambda}\\
&&\mathrm{s.t.}&p-\lambda= \sum_j
b^*_jb_j+\sum_{i=1}^{m}\sum_j
c^*_{ij}q_ic_{ij}\\
&&& \max_{j}\text{deg}(b_j)\leq k,\\
&&&\max_{j}\text{deg}(c_{ij})\leq k-d_i\,.
\end{array}
\end{equation}
In the case of polynomials defined on $\mathbb{C}[x,x^*]$, the dual of (\ref{relax}) has the same form as above, but now all polynomials are allowed to take complex coefficients.

A similar analysis can be carried to show that the problems (\ref{relaxb}) and (\ref{dual2gen}) are dual to each other.

\addcontentsline{toc}{section}{References}
\bibliographystyle{unsrt}
\bibliography{poi_relaxation}

\end{document}